\documentclass[preprint,12pt,times]{elsarticle}

\usepackage[utf8]{inputenc}             
\usepackage[T1]{fontenc}                
\usepackage{natbib}
\usepackage[french,english]{babel}              
\usepackage{graphicx}                   
\usepackage{amsmath,amsfonts,amssymb}   
\usepackage{amsthm}
\usepackage{mathtools}
\mathtoolsset{showonlyrefs}
\usepackage{bm}
\usepackage{geometry}
\usepackage{algorithm}
\usepackage{algorithmic}
\usepackage{hyperref}
\usepackage{appendix}
\usepackage{color}
\usepackage{enumitem}
\usepackage{bm}
\usepackage{mathtools}
\usepackage{nccmath}
\usepackage{yhmath}
\usepackage{tikz}
\usetikzlibrary{arrows,decorations.pathmorphing,shapes}

 \makeatletter
 \def\ps@pprintTitle{%
  \let\@oddhead\@empty
  \let\@evenhead\@empty
  \def\@oddfoot{}%
  \let\@evenfoot\@oddfoot}
 \makeatother

\journal{Pattern Recognition}

\DeclareMathOperator{\id}{id}

\DeclareMathOperator{\R}{\mathbb{R}}

\newcommand{\dd}{\mathrm{d}}

\newtheorem{theo}{Theorem}[section]
\newtheorem{remark}{Remark}[section]

\newtheorem{lemma}{Lemma}[section]
\newtheorem{cor}{Corollary}[section]

\usepackage{chngcntr}
\usepackage{apptools}
\AtAppendix{\counterwithin{lemma}{section}}

\begin{document}

\begin{frontmatter}



\title{Elastic Metrics on Spaces of Euclidean Curves: Theory and  Algorithms}


\author[label1]{Martin Bauer}
\author[label2]{Nicolas Charon}
\author[label1]{Eric Klassen}
\author[label3]{Sebastian Kurtek}
\author[label1]{Tom Needham}
\author[label1]{Thomas Pierron}

\address[label1]{Florida State University, Department of Mathematics, Tallahassee, Florida, USA}
\address[label2]{Johns Hopkins University, Department of Applied Mathematics and Statistics, Baltimore, Maryland, USA}
\address[label3]{Ohio State University, Department of Statistics, Columbus, Ohio, USA}


\begin{abstract}
   A main goal in the field of statistical shape analysis is to define computable and informative metrics on spaces of immersed manifolds, such as the space of curves in a Euclidean space. The approach taken in the elastic shape analysis framework is to define such a metric by starting with a reparameterization-invariant Riemannian metric on the space of parameterized shapes and inducing a metric on the quotient by the group of diffeomorphisms. This quotient metric is computed, in practice, by finding a registration of two shapes over the diffeomorphism group. For spaces of Euclidean curves,  the initial Riemannian metric is frequently chosen from a two-parameter family of Sobolev metrics, called elastic metrics. Elastic metrics are especially convenient because, for several parameter choices, they are known to be locally isometric to Riemannian metrics for which one is able to solve the geodesic boundary problem explictly---well-known examples of these local isometries include the complex square root transform of Younes, Michor, Mumford and Shah and square root velocity (SRV) transform of Srivastava, Klassen, Joshi and Jermyn. In this paper, we show that the SRV transform extends to elastic metrics for all choices of parameters, for curves in any dimension, thereby fully generalizing the work of many authors over the past two decades. We give a unified treatment of the elastic metrics: we extend results of Trouv\'{e} and Younes, Bruveris as well as Lahiri, Robinson and Klassen on the existence of solutions to the registration problem, we develop algorithms for computing distances and geodesics, and we apply these algorithms to metric learning problems, where we learn optimal elastic metric parameters for statistical shape analysis tasks. 
\end{abstract}


\end{frontmatter}


\section{Introduction}

Shape is a fundamental physical property of objects and a key characteristic of their appearance in images. As a result, shape analysis plays a central role in various applications including computer vision, medical imaging, graphics, biology, bioinformatics and anthropology, among others. In these applications, one generally first extracts objects of interest from the imaging data, and then studies their shapes via appropriate mathematical representations and metrics. In statistical shape analysis, each observed shape is treated as a random object with the primary goal of developing tools for shape registration, comparison, statistical summarization, exploration of variability, clustering, classification and other statistical procedures. Each of the aforementioned statistical tasks heavily depends on the underlying representation and associated metric chosen for shape analysis.

There is a rich literature on shape analysis that considers various representations of shape including deformable templates \cite{grenander-miller:98}, ordered and unordered point sets \cite{dryden-mardia_book:98}, level sets \cite{osher-fedkiw}, medial axes \cite{Gorczowski}, and others. However, perhaps the most natural representation of a boundary of an object captured in an image is a parameterized curve. While accounting for the shape preserving transformations of rigid motion and scaling is fairly standard in this setting \cite{srivastava_2016}, one must additionally deal with parameterization variability inherent in the given data. Some past methods standardize parameterizations of observed curves to arc-length \cite{zahn-roskies:72}, but this has been shown to be suboptimal in many applications \cite{srivastava_2016}. A better solution is to determine optimal reparameterizations in a pairwise manner via a process referred to as \emph{registration}. This, in turn, requires a metric on the space of parameterized curves that is invariant to reparameterizations. The metric plays a key role in shape analysis as it is used for joint registration and comparison of shapes. Further, it serves as a backbone of other statistical procedures for shape data including averaging and principal component analysis.

In this article, we focus on shapes which are represented as curves in Euclidean space $\R^d$, $d \geq 2$. Our shape metrics arise as geodesic distances with respect to Riemannian metrics on the (infinite-dimensional) manifold $\mathcal I(D,\R^d)$, whose points are immersions, with domain $D$ either an interval or a circle. That is, each element of $\mathcal I(D,\R^d)$ is a smooth parameterized curve $c:D \to \R^d$ with nowhere-vanishing derivative. In order to induce a metric on the shape space of unparametrized curves, one requires the Riemannian metric on $\mathcal I(D,\R^d)$ to be invariant under reparameterizations. To be precise, the group $\mathcal D(D)$ of diffeomorphisms of $D$ acts on $\mathcal I (D,\R^d)$ by precomposition, and this action should be by isometries for the chosen Riemannian metric. Due to this property, the metric then descends to the quotient space $\mathcal I(D,\R^d)/\mathcal D(D)$ of curves considered up to precomposition with a diffeomorphism---that is, the quotient space can be considered as the space of unparameterized curves. The process of computing geodesic distances in that quotient space naturally involves solving a registration problem, so that the estimated geodesic distance eventually provides a meaningful metric for shape comparison. Moreover, the Riemannian formalism gives powerful tools for demonstrating the existence of optimal registrations, along with well-defined notions of tangent spaces, means, principal components, and other statistical concepts.

In this setup, a variety of different Riemannian metrics have been proposed in the literature. The arguably simplest one, the invariant $L^2$-metric, has a surprising degeneracy: it induces vanishing geodesic distance on both parametrized \cite{bauer2012vanishing} and unparametrized \cite{mumford2006riemannian} curves; i.e., any two curves (shapes) are regarded the same under the corresponding path-length distance. This behavior renders the $L^2$-metric unsuitable for any applications in shape analysis. Subsequently, several stronger Riemannian metrics have been proposed, that consequently induce a meaningful measure of similarity on shape space. This includes the class of almost local metrics \cite{mumford2006riemannian}, but also the family of (higher order) Sobolev type metrics, see e.g.~\cite{michor2007overview,bauer2014overview,srivastava2010shape,sundaramoorthi2007sobolev} and the references therein. In particular, Sobolev metrics of order one have attracted a large body of work and a two-parameter family of Riemannian metrics $G^{a,b}$, $a,b > 0$, has been proposed \cite{mio2007shape}; elements of this family are usually called \emph{elastic metrics}, for reasons that are highlighted in \ref{subsec:elasticity}. 
The goal of this paper is to develop a comprehensive theoretical and computational framework for the $G^{a,b}$ metrics on $\mathcal I(D,\R^d)$ and the quotient $\mathcal I(D,\R^d)/\mathcal D(D)$, for all parameters $a,b > 0$ and all dimensions $d \geq 2$. Before precisely stating our main contributions, we first give an overview of related work to provide appropriate context.

A crucial component of an efficient algorithm for computing the desired geodesic distances in the quotient space $\mathcal I (D,\R^d)/\mathcal D(D)$ is a method for computing distances in the  space $\mathcal I (D,\R^d)$. The family of elastic metrics is special: it has been shown that, for several values of the parameters $a$, $b$ and $d$, the geodesic boundary value problem, and consequently the induced geodesic distance, can be solved \emph{explicitly}. The typical method in the literature for deriving such explicit solutions is to construct an isometry $(\mathcal I (D,\R^d), G^{a,b}) \to (\mathcal M, G)$, for particular values of $a$, $b$ and $d$, to some Riemannian manifold with explicit formulas for geodesics, and to then describe geodesics in the source space via pulling back. For the choice of parameters $a=1$ and $b=\frac{1}{2}$, such an isometry is given by the well known \emph{square root velocity transform}, as developed in \cite{srivastava2010shape} for arbitrary $d \geq 2$. For $a=b=\frac{1}{2}$ and $d=2$, an isometry is given by the \emph{complex square root mapping}  constructed in \cite{younes2008metric}, which is based on identifying $\R^2$ with the complex plane. The complex square root mapping was generalized to curves in $\R^3$ by replacing complex constructions with their quaternionic counterparts in \cite{needham2018kahler, needham2019shape}.
For $a\leq 2b$ and $d=2$, a related construction has been developed in~\cite{bauer2014constructing}, where the target manifold is a space of curves in a Euclidean cone. Finally, for curves with values in $\mathbb R^2$, the transformations of \cite{younes2008metric,srivastava2010shape,bauer2014constructing} have been extended to general values of the parameters $a$ and $b$ using a local isometry defined, once again, via the identification of $\R^2$ with the complex plane \cite{needham2020simplifying}.

We can now state our main contributions and outline the structure of the paper.
\begin{itemize}
    \item {\bf Simplifying isometry for general elastic metric parameters (Section~\ref{sec:elastic_metric}).} We show that the square root velocity transform mentioned above can, in fact, be used as a simplifying isometry of $G^{a,b}$ for  general parameters $a$ and $b$ and for curves in arbitrary dimension (Theorem~\ref{genSRVT}), thereby fully generalizing the results of \cite{srivastava2010shape,bauer2014constructing,needham2020simplifying} and completing the story started over a decade ago in \cite{younes2008metric}. We use this result to  characterize the metric completion of the space of curves  and give an explicit formula for the distance in this space (Theorem~\ref{thm:completion}). This completion includes curves of lower regularity---in particular, it includes the class of piecewise linear curves, which is important for representing smooth curves in a discrete computational setting. We also present a (to our knowledge, novel) relation between elastic metrics and classical elasticity theory (\ref{subsec:elasticity}). 
    \item {\bf Existence of optimal reparameterizations (Section~\ref{sec:optimal_reparameterizations}).} To obtain the distance on the quotient shape space, one needs to consider the following problem: given two curves $c_1,c_2 \in \mathcal I(D,\R^d)$, find a reparameterization $\gamma \in \mathcal D(D)$ such that the geodesic distance between $c_1$ and $c_2 \circ \gamma$, with respect to a given elastic metric $G^{a,b}$, is minimized. It has been shown that a minimizer exists (within certain extensions of the set $\mathcal D(D)$) for parameters $a=1$ and $b=\frac{1}{2}$ (the original setting of the square root velocity transform), under certain regularity assumptions on the curves $c_i$ \cite{lahiri2015precise,bruveris2016optimal}. We extend these results to general parameters $a,b > 0$, under  technical regularity assumptions (Theorems~\ref{thm:existence}~and~\ref{thm:existenceclosed}) and lay out some precise open questions regarding dependence on regularity. Our proof uses classical results of Trouv\'{e} and Younes~\cite{article}
    \item {\bf Computational framework (Section~\ref{sec:algorithms}).} We develop a comprehensive  framework for computing geodesics in the quotient space $\mathcal I (D,\R^d)/\mathcal D(D)$. This computation involves optimizing geodesic distance over reparameterizations, as was described in the previous paragraph. We give an explicit polynomial-time algorithm to find exact solutions in the setting of piecewise linear curves (Theorem \ref{thm:explicit_algorithm}) and a faster dynamic programming algorithm for approximating the optimizer (Section~\ref{sec:dynamic_programming}). We illustrate this approach with several computational examples (Section~\ref{sec:examples}). Our code is available under an open source license\footnote{\url{https://github.com/charoncode/Gab_metrics}}.
    \item {\bf Metric Learning (Section~\ref{sec:metric_learning}).} Finally, we consider the following question: given a dataset of shapes, which metric from the family of elastic metrics gives the best performance on various statistical analysis tasks? We frame this as a \emph{metric learning} problem. A general approach to learning the appropriate parameters for a given dataset is suggested (Section \ref{ssec:pairwise_constraints}), based on foundational work in metric learning~\cite{xing2002distance}. We also give an alternative approach to parameter estimation with the aim of training for geometric protein classification (Section~\ref{ssec:shape_classification}). 
\end{itemize}

\subsubsection*{Acknowledgements}
M. Bauer has been supported by NSF-grants DMS-1912037 and DMS-1953244 and by the Austrian Science Fund grant P 35813-N, N. Charon by NSF grants DMS-1953267 and DMS-1945224, S. Kurtek by NSF grants CCF-1740761, CCF-1839252 and DMS-2015226, and NIH grant R37-CA214955, T. Needham by NSF grant DMS--2107808 and T.  Pierron by NSF-grant DMS-1912037.

\section{Simplifying Transform for General Elastic Metrics}\label{sec:elastic_metric}
In this section, we will introduce the basic concepts and spaces under consideration and introduce the class of Riemannian metrics
that will be of central interest. We then define a new family of transforms for simplifying these Riemannian metrics.

\subsection{Spaces of curves and elastic metrics}
\label{ssec:elastic_metrics}

In the following, assume $d\geq 2$ and let
\begin{equation}
\mathcal I(D,\R^d) = \{ c \in C^\infty(D,\R^d) \,:\, c'(u) \neq 0\; \forall u \in D\}\,,
\end{equation}
where $D$ is either the interval $I=[0,1]$ for open curves or the unit circle $S^1$ for closed curves.
This space is an open subset of the Fr\'{e}chet space $C^\infty(D,\R^d)$, and is thus an infinite-dimensional manifold with tangent space $T_c\mathcal I(D,\mathbb R^d)$ at a curve $c \in \mathcal{I} (D,\R^d)$ satisfying $T_c\mathcal I(D,\mathbb R^d) \approx C^{\infty}(D,\mathbb R^d)$; specifically, the identification is made by taking $C^\infty(D,\R^d)$ to be the space of smooth vector fields (or \emph{deformation fields}) along $c$.

This article is concerned with the family of elastic $G^{a,b}$-metrics on $\mathcal{I}(D,\R^d)$, indexed over pairs of positive constants $a$, $b > 0$:
\begin{align}
\label{eq:our_metric_ab}
G^{a,b}_c(h,k) = \int_D  a^2 ( D_s h^{\bot} \cdot D_s k^{\bot} )  + 
b^2  
 (D_s h^{\top} \cdot D_s k^{\top}) 
ds\,,
\end{align}
where $h,k \in T_c \mathcal I(D,\R^d)$ are deformation fields (tangent vectors) to the curve $c\in  \mathcal I(D,\R^2)$, $\cdot$ denotes the Euclidean inner product with norm $|\cdot |$ (evaluated pointwise),
$D_s = \frac{1}{|c'|} \frac{d}{du}$ and $ds = |c'| du$ are differentiation and integration with respect to arc length, respectively ($u$ denoting the parameter of $c$), and $\bullet^\top$ and $\bullet^\bot$ denote projection onto the normal and tangential part of a tangent vector, i.e.,
\begin{equation}
D_s h^{\top}=  \left(D_s h \cdot \tfrac{c'}{|c'|}\right) \tfrac{c'}{|c'|} \qquad \mbox{and} \qquad D_s h^{\bot}=D_s h-D_s h^{\top}.
\end{equation}
The terminology of "elastic metrics" for \eqref{eq:our_metric_ab}  often used in the literature \cite{younes1998computable,mio2007shape,jermyn2012elastic,needham2019knot} can be in fact justified from the theory of linear material elasticity,  specifically as the limit of the linear elastic energy of a deforming shell as it becomes infinitely thin. Such a connection was recently emphasized in \cite{charon2022shape} for the class of first order metrics on surfaces. We provide in \ref{subsec:elasticity} a similar and more direct derivation in the case of parametrized planar curves.

The group $\R^d$ acts on $\mathcal{I}(D,\R^d)$ by rigid translations. Each bilinear form \eqref{eq:our_metric_ab} is degenerate on the space of all curves and therefore only defines a Riemannian metric on the quotient $\mathcal I(D,\R^d)/\R^d$, which can be identified with the space $\mathcal I_0(D,\R^d)$ of curves starting at the origin. Once we have defined a Riemannian metric, we can consider the corresponding geodesic distance function
\begin{equation}
d_{a,b}(c_0,c_1)=\operatorname{inf} \int_0^1 \sqrt{G_c(c_t,c_t)} dt,
\end{equation}
where the infimum is taken over all paths $c:[0,1]\to \mathcal I_0(D,\R^d):t \mapsto c_t$ interpolating between the curves $c_0$ and $c_1$. For finite-dimensional Riemannian manifolds, geodesic distance is indeed a true metric, but this is not necessarily true in infinite dimensions: there are Riemannian  metrics such that the corresponding geodesic distance function is degenerate or might even vanish identically~\cite{mumford2006riemannian,bauer2020vanishing}. For the $G^{a,b}$-metrics this misbehavior has been ruled out~\cite{mumford2006riemannian,srivastava2010shape}, which consequently renders them as viable candidates for shape analysis.

On the space of immersions, there is a natural action by the orientation-preserving diffeomorphism group $\mathcal D(D)$ of the domain $D$: the reparametrization action. Given a curve $c\in \mathcal I_0(D,\R^d)$ and a diffeomorphism $\varphi\in \mathcal D(D)$ this action is given by composition from the right, i.e.,
\begin{equation}
\mathcal I_0(D,\R^d)\times \mathcal D(D)\to  \mathcal I_0(D,\R^d),\qquad (c,\varphi)\mapsto c\circ\varphi.
\end{equation}
A straightforward calculation shows that the $G^{a,b}$-metrics are invariant under this action, i.e.,
\begin{equation}
G^{a,b}_c(h,k)=G^{a,b}_{c\circ\varphi}(h\circ\varphi,k\circ\varphi), \qquad h,k\in T_c  \mathcal I_0(D,\R^d),\quad \varphi\in \mathcal D(D).  
\end{equation}
Consequently, they descend to Riemannian metrics on the quotient shape space of immersions modulo parametrizations $\mathcal S(D,\mathbb R^d):=\mathcal I_0(D,\R^d)/\mathcal D(D)$. On the quotient space, the corresponding geodesic distance function can be calculated via
\begin{equation}
d^{\mathcal S}_{a,b}([c_0],[c_1])=\underset{\varphi \in \mathcal D(D)}{\operatorname{inf}}d_{a,b}(c_0,c_1\circ \varphi).
\end{equation}

\subsection{The square root velocity transform for general elastic metrics}

As we overviewed in the introduction, there have been many approaches in the literature to understanding elastic metrics through \emph{simplifying transformations} \cite{younes2008metric,srivastava2010shape,bauer2014constructing,needham2020simplifying}. That is, these works establish (local) isometries of the form $(\mathcal I (D,\R^d), G^{a,b}) \to (\mathcal M, G)$, for some choice of parameters $a$, $b$, and $d$, where the target space is some Riemannian manifold with an easy-to-describe geodesic structure. Such a transformation allows efficient computations involving the elastic metric by transferring them to the simple target space. Of particular interest for this paper is the \emph{square root velocity transform} of Srivastava et al., which we denote as
\begin{equation}\label{eqn:SRVT}
  \begin{aligned}
R: \mathcal I_0([0,1],\mathbb R^d) &\to C^{\infty}([0,1],\mathbb R^d\setminus\{0\})  \\
c&\mapsto \frac{c'}{\sqrt{|c'|}}.
  \end{aligned}
  \end{equation}
It was shown in~\cite{srivastava2010shape} that $R$ is an isometry of the elastic metric $G^{1,\frac{1}{2}}$ and the standard $L^2$ metric on $C^{\infty}([0,1],\mathbb R^d\setminus\{0\})$, for any $d \geq 2$. Our first main result below will show that $R$ is, for general parameters $a,b$, an isometry of $G^{a,b}$ and a Riemannian metric on $C^{\infty}([0,1],\mathbb R^d\setminus\{0\})$ which is non-Euclidean, but still simple enough to admit explicit geodesic distances.

In order to formulate this result, we first introduce a Riemannian metric on $\mathbb R^d\setminus\{0\}$. For  $q\in \R^d\setminus\{0\}$, we identify $T_q\left(\R^d\setminus\{0\}\right)$ with $\R^d$ in the obvious way. We then decompose the tangent space via $T_q\left(R^d\setminus\{0\}\right)=V\oplus W$, where $V=\R q$ and $W=V^{\bot}$, where the orthogonal complement is with respect to the standard dot product on $\R^d$. We then define a Riemannian metric $g^{\lambda}$ on $T_q\left(\R^d\setminus\{0\}\right)$ as follows:
\[
g^\lambda_q(v,w)=\lambda^2  (v^{\bot}\cdot w^{\bot}) +  (v^{\top}\cdot w^{\top}) ,
\]
 where, in analogy with the notation used in Section \ref{ssec:elastic_metrics}, $w^{\top}=\frac{( w\cdot q)}{( q\cdot q)}q$ is the projection of $w$ onto $V$ and $w^{\bot}=w-w^{\top}$ is the projection of $w$ onto $W$.  For $\lambda=1$, this is just a restriction of the standard Euclidean metric on $\R^d$; for $\lambda<1$, it makes $\R^d \setminus \{0\}$ isometric to a dense subset of a cone in $\R^{d+1}$ with an acute angle at the cone point; for $\lambda>1$, $(\R^d \setminus \{0\}, g^\lambda)$ does not isometrically embed in $\R^{d+1}$. 
 
 The metric $g^{\lambda}$ induces an $L^2$-metric on the space of smooth curves in  $\R^d \setminus \{0\}$:
 \begin{align}
G^{L^2_\lambda}_q(q_1,q_2)=\int_0^1 g^{\lambda}_q(q_1,q_2)du.
 \end{align}
We now state our first main result, whose proof is postponed to  \ref{appendix:existenceopen}.
\begin{theo}\label{genSRVT}
For $\lambda=\tfrac{a}{2b}$, the square root velocity transform $R$, defined in \eqref{eqn:SRVT}, is an isometry of $\left(\mathcal I_0([0,1],\mathbb R^d),G^{a,b}\right)$ and $\left(C^{\infty}([0,1],\mathbb R^d\setminus\{0\}),4b^2 G^{L^2_{\lambda}}\right)$.
Furthermore, for each $c_0\in \mathcal I_0([0,1],\mathbb R^d)$ there exists a neighborhood $\mathcal U(c_0)$ such that the geodesic distance between $c_0$ and any $c_1\in \mathcal U(c_0)$ is given by:
\begin{equation}\label{eq:dist}
\operatorname{dist}_{a,b}(c_1,c_2)=2b\sqrt{\ell_{c_1}+\ell_{c_2}-2 \int_D  \sqrt{|c_1'||c_2'|}\cos\left(\tfrac{a}{2b}\theta\right) du},
\end{equation}
where 
\begin{equation}
\theta(u)=\cos^{-1}(R(c_1)\cdot R(c_2)/|R(c_1)||R(c_2)|). 
\end{equation}
If $d\geq 3$, then the formula for the geodesic distance holds globally, i.e., for arbitrary $c_1\in \mathcal I_0([0,1],\mathbb R^d)$, after replacing the formula for $\theta$ by
\begin{equation}
\theta(u)=
\operatorname{min}\left(\cos^{-1}(R(c_1)\cdot R(c_2)/|R(c_1)||R(c_2)|),\tfrac{2b\pi}{a}\right). 
\end{equation}
\end{theo}

To deal with the difficulty in the $d = 2$ case (the formula of the geodesic distance being only valid locally), we can extend geodesics across
the origin and obtain $C^\infty(I, \mathbb R^d_\lambda)$ as the geodesic completion in the sense of \cite{khesin2004flow,bruveris2016optimal}. This allows us to interpret \eqref{eq:dist} as the geodesic distance on the geodesic completion. We will now extend the formula for the geodesic distance of Theorem \ref{genSRVT} to the metric completion, which will be important in the next section where we will prove the existence of optimal reparametrizations.
\begin{cor}\label{thm:completion}
The completion (in the sense of Lemma~\ref{lem:complL2}) of $\mathcal I_0([0,1],\mathbb R^d)$ with the $G^{a,b}$ metric is the space of absolutely continuous open curves $AC_0([0,1],\mathbb R^d)$. For any two curves $c_1,c_2\in AC_0([0,1],\mathbb R^d)$ their corresponding geodesic distance is given by
\begin{equation}\label{eq:dist2}
\operatorname{dist}_{a,b}(c_1,c_2)=2b\sqrt{\ell_{c_1}+\ell_{c_2}-2 \int_0^1  \sqrt{|c_1'||c_2'|}\cos\left(\tfrac{a}{2b}\theta\right) du}.
\end{equation}
where $\ell_{c_j}$ is the length of the curve $c_j$ and
\begin{equation}
\theta(u)=\begin{cases}
\operatorname{min}\left(\cos^{-1}(c'_1(u)\cdot c'_2(u)/|c'_1(u)||c'_2(u)|),\tfrac{2b\pi}{a}\right) &\text{ if $c'_1(u), c'_2(u) \neq 0$}\\
 \tfrac{2b\pi}{a} &\text{ otherwise.}
\end{cases}
\end{equation}
\end{cor}
 
\subsection{Relation to previous work}\label{sec:relation_to_previous_work}

In this subsection, we pin down the precise relationship between Theorem \ref{genSRVT} and previous work on simplifying transforms~\cite{younes2008metric,srivastava2010shape,bauer2014constructing,needham2020simplifying}.

The transform in the literature which is most relevant to our result is obviously the square root velocity transform \eqref{eqn:SRVT}, which was shown in \cite{srivastava2010shape} to be an isometry of $G^{1,\frac{1}{2}}$ and the standard $L^2$ metric on $C^\infty([0,1],\R^d \setminus \{0\})$ for arbitrary $d \geq 2$. This result is recovered directly from Theorem~\ref{genSRVT}.

The \emph{complex square root map} of~\cite{younes2008metric} takes an immersion $c$ in the plane to the curve $\sqrt{c'}$, where the square root is computed pointwise by considering $c$ as a complex-valued function---there is some ambiguity here, so the square root curve is chosen in a way to make it continuous. This transform was shown to be a local isometry of $G^{\frac{1}{2},\frac{1}{2}}$ with the standard $L^2$ metric on $C^\infty([0,1],\mathbb{C}) \approx C^\infty([0,1],\R^2)$. In~\cite{needham2020simplifying}, it was shown that the complex square root map fits into a family of maps $F_{a,b}$, defined on a smooth plane curve $c$ by $F_{a,b}(c) = 2b|c'|^{\frac{1}{2}} \left(c'/|c'|\right)^{\frac{a}{2b}}$, with exponentiation once again performed using the identification of $\R^2$ with the complex plane and choosing a continuous curve as the image; when $a = b = \frac{1}{2}$, $F_{a,b}$ reduces to the complex square root map. It was shown in~\cite{needham2020simplifying} that $F_{a,b}$ defines a local isometry between $G^{a,b}$ and the $L^2$ metric, for any choices of $a,b > 0$. In fact, the $F_{a,b}$ transform factors as

\begin{tikzpicture}[->,>=stealth',auto,node distance=5.2cm,
  thick,main node/.style={}]
  \hspace{-.4cm}
  \node[main node] (1) {$\left(\mathrm{Imm}_0([0,1],\mathbb{C}),G^{a,b}\right)$};
  \node[main node] (2) [right of=1] {$\left(C^\infty([0,1],\mathbb{C}^\ast),4b^2G^{L^2_\lambda}\right)$};
  \node[main node] (3) [right of=2] {$\left(C^\infty([0,1],\mathbb{C}^\ast),G^{L^2}\right)$};

  \path[every node/.style={font=\sffamily\small}]
    (1) edge node [right,midway,above] {$R$} (2)
    (2) edge node [right,midway,above] {$S_{a,b}$} (3)
    (1) edge[bend right=15] node [left,midway,above] {$F_{a,b}$} (3);
\end{tikzpicture}

\noindent where $\mathbb{C}^\ast = \mathbb{C} \setminus \{0\}$,  $\lambda = \frac{a}{2b}$, $G^{L^2} = G^{L^2_1}$ is the standard $L^2$ metric and $S_{a,b}$ is a local isometry defined by $S_{a,b}(q) = 2 b |q|^{1-\lambda} q^\lambda$---the local isometry claim can be seen via calculations in complex coordinates, similar to the proof of~\cite[Theorem 2.3]{needham2020simplifying}. The takeaway from Theorem~\ref{genSRVT} is that the geometry of the metric $G^{L^2_\lambda}$ is simple enough that we can work in the middle space of this diagram, allowing us to avoid technical issues with isometries only being local, while simultaneously allowing the result to be generalized to arbitrary dimension.

We should also mention~\cite{bauer2014constructing}, which gave a similar family of generalizations of the complex square root map for plane curves, valid for $G^{a,b}$ with $a \leq 2b$. In this case, a curve $c$ in $\mathbb{C}$ is mapped to the curve in $\R^3 \approx \mathbb{C} \times \R$ given by 
\[
R_{a,b}(c) = |c'|^{\frac{1}{2}} \left(a\frac{c'}{|c'|}, \sqrt{4b^2 - a^2}\right).
\]
The image of $R_{a,b}$ lies on a certain cone in $\R^3$ and it is shown in~\cite{bauer2014constructing} that the transform is an isometry of $G^{a,b}$ and the metric on the cone induced from the ambient Euclidean metric. As in the case of the $F_{a,b}$ transform described above, one can factor $R_{a,b}$ as the SRV transform $R$ followed by an isometry between the space of curves in the plane and the space of curves in the cone. 

\section{Existence of Optimal Reparameterizations}\label{sec:optimal_reparameterizations}

\subsection{Existence of optimal reparametrizations for open curves}
In the previous section we saw that the set of absolutely continuous functions provides the natural space for studying the geodesic distance function of the family of elastic metrics. In the following, we aim to prove the existence of optimal reparametrizations in this space.  Let $\operatorname{dist}^{\mathcal S}_{a,b}$ be the induced distance of the elastic $G^{a,b}$-metric on the space of absolutely continuous, unparametrized and open curves $\mathcal S(I,\mathbb R^d):=\operatorname{AC}_0(I,\mathbb R^d)/\Gamma$, which is defined by
\begin{equation}
\operatorname{dist}^{\mathcal S}_{a,b}([c_1],[c_2])=\operatorname{inf}_{\gamma \in \Gamma}\operatorname{dist}_{a,b}(c_1,c_2\circ \gamma),\label{quotient-distance1}
\end{equation}
where $\operatorname{dist}_{a,b}$ denotes the geodesic distance function on the space of parmetrized open curves $\operatorname{AC}_0(I,\mathbb R^d)$, and where $\Gamma$ denotes the group of absolutely continuous diffeomorphisms on $I$, i.e.,
 \begin{align*}
\Gamma=\{\gamma \in AC(I,I): \ \gamma(0)=0, \ \gamma(1)=1, \ \gamma'>0\quad a.e. \}. 
 \end{align*}
We will also need the closure (with respect to the norm topology) of this group $\bar \Gamma$, which is the semigroup
 of weakly increasing, absolutely continuous functions, i.e.,
 \begin{align*}
\bar \Gamma=\{\gamma \in AC(I,I): \ \gamma(0)=0, \ \gamma(1)=1, \ \gamma'\geq 0\quad a.e. \}. 
\end{align*}
With this notation, the induced distance
$\operatorname{dist}^{\mathcal S}_{a,b}$ on the quotient space $\mathcal S(I,\mathbb R^d)$ can be equivalently expressed via
\begin{align}
\label{quotient-distance2}
\operatorname{dist}^{\mathcal S}_{a,b}([c_1],[c_2])&=\inf_{\gamma\in \Gamma}\operatorname{dist}_{a,b}(c_1,c_2\circ \gamma) 
=
\inf_{\gamma_1,\gamma_2\in \Gamma}\operatorname{dist}_{a,b}(c_1\circ\gamma_1,c_2\circ \gamma_2) \nonumber\\
&=
\inf_{\gamma_1,\gamma_2\in \bar \Gamma}\operatorname{dist}_{a,b}(c_1\circ\gamma_1,c_2\circ \gamma_2).
\end{align}
Here, the second equality in the first line follows from the invariance of the distance, whereas the third equality follows from the density of $\Gamma$ in $\bar \Gamma$.
Our main result of this section concerns the existence of optimal reparametrizations, i.e., the existence of reparametrization functions such that the infimum is attained. We will see that for $a\leq b$ we really need two reparametrization functions in $\bar \Gamma$, while for $a>b$ the infimum can attained by one reparametrization function. Before we formulate the theorem we need to introduce the function space of piecewise differentiable functions: 
\begin{align*}
PC^1(I,\mathbb R^d):= &\left\{c\in C(I, \mathbb R^d): \exists \;  0=t_0<t_1<\ldots<t_n=1\right. \\
&\qquad \qquad \qquad \qquad \left.\text{ s.t. } c|_{(t_i,t_{i+1})}\in C^1\big((t_i,t_{i+1}),\R^d\big)  \right\}.
\end{align*}
Note that $PC^1(I,\mathbb R^d)\subset AC(I,\mathbb R^d)$. We also need to introduce the space $\mathcal{D}^*$ of all functions that can be written as
\begin{equation}
\phi(s) = \mu([0,s[)
\label{eq:def_mu}
\end{equation}
where $\mu$ is a probability measure on $I$. These are equivalently (c.f. \cite{cohn2013measure}, Chapter 4) bounded variation (BV) functions, which are non-decreasing and left-continuous on $I$. 
We will denote the right limit of $\phi$ at $x$ by $\phi(x+0^+)$. 
Furthermore, we recall that since $\phi$ is nondecreasing, $\phi$ is differentiable almost everywhere. 
For $\phi \in \mathcal{D}^*$, we also introduce a generalized inverse $\phi^- \in \mathcal{D}^*$ defined by:
$$
\phi^-(y) = \sup\{x \in [0,1], \ \phi(x)<y\},
$$
where we use the convention $\sup \emptyset = 0$; c.f.\ \cite[Section 5.2.3]{article}. 

We are now able to formulate the main result of this section.
\begin{theo}\label{thm:existence}
Let $c_1,c_2\in PC^1(I,\mathbb R^d)$. Then the distance $\operatorname{dist}^{\mathcal S}_{a,b}$ on the quotient space $\mathcal S(I,\mathbb R^d)$ is equivalently given by a supremum over the space $\mathcal D^*$, i.e.,
\begin{align}\label{quotien_distance_BV}
\operatorname{dist}^{\mathcal S}_{a,b}([c_1],[c_2])&=
2b\sqrt{\ell_{c_1}+\ell_{c_2}-2\sup_{\phi\in \mathcal D^*}\int_I\sqrt{\dot\phi(x)} f_{a,b}\left(x,\phi(x)\right) d x},
\end{align}
where $f_{a,b} : I\times I  \rightarrow  \mathbb{R}$ is defined by
\begin{equation}\label{eq:f_ab}
f_{a,b}(x,y )=\left\{ \begin{array}{cl}
          \sqrt{|\dot{c}_1(x)| |\dot{c}_2(y)|}\cos\left(\frac{a}{2b} \cos^{-1}(\frac{\dot{c}_1(x)\cdot \dot{c}_2(y)}{|\dot{c}_1(x)| |\dot{c}_2(y)|})\right) & \mbox{ if } \frac{a}{2b} \cos^{-1}(\frac{\dot{c}_1(x)\cdot \dot{c}_2(y)}{|\dot{c}_1(x)| |\dot{c}_2(y)|}) \leq \frac{\pi}{2}   \\
          0 & \mbox{ otherwise.}  
     \end{array} \right. 
\end{equation}

\noindent
We have the following statements concerning the existence of optimal reparametrizations:
\begin{enumerate}
    \item $\mathbf{a<b:}$ for any $c_1,c_2\in PC^1(I,\mathbb R^d)$ there exists a strictly increasing homeomorphism  such that the infimum in~\eqref{quotient-distance1} is attained. If the derivatives $\dot c_1, \dot c_2$ are Lipschitz continuous then $\gamma \in \operatorname{Diff}_{C^1}(I)$, i.e., the optimal reparametrization is a $C^1$-diffeomorphism. 
    \item $\mathbf{a\geq b:}$ for  $c_1,c_2\in PC^1(I,\mathbb R^d)$ there exists a pair of generalized reparametrizations $\gamma_1,\gamma_2\in \bar \Gamma$ such that the infimum in~\eqref{quotient-distance2} is attained. 
    On the other hand, there exists a pair of curves $c_1,c_2\in AC(I,\mathbb R^d)$ such that the infimum in~\eqref{quotient-distance2} is not attained in $\bar \Gamma$ and consequently neither in $\Gamma$ or $\mathcal D^*$.
\end{enumerate}
\end{theo}
\noindent
Our proof of this result, which makes repeated use of results by Trouve and Younes~\cite{article} is presented in \ref{appendix:existenceopen}.

\begin{remark}[Open questions]
This result suggests the following questions, which remain open for future research.
\begin{itemize}
    \item {\bf Counterexample for} $a<b$: The proof of the non-existence result for $a\geq b$ will be based on constructing curves $c_1,c_2$ such that $f_{a,b}(x,x)\leq 0$ for $x\in B$, where $B$ is a closed but nowhere dense subset. For $a<b$,  $f_{a,b}$ is positive and thus the same strategy fails. 
    \item {\bf Higher regularity:} One would hope that a higher regularity of the  curves $c_1$ and $c_2$ would lead to a higher regularity of the obtained optimal reparametrization functions. To deduce this result  from the theorem of Trouve and Younes~\cite{article}  one would need that a higher regularity of the curves $c_i$ also leads to a higher regularity of the function $f_{a,b}$. This function is, however, at best Lipschitz continuous and only locally of a higher regularity. One can use the local regularity of $f_{a,b}$ and localize the arguments of~\cite{article}to show that the optimal reparametrization functions are locally of class $C^{k-1}$ provided that the curves are of class $C^k$, but as of now we do not know how one could go a step further and obtain a global regularity result.
\end{itemize}
\end{remark}
\begin{remark}
In the limit (and degenerate) case $a=0$, one can further show that for any regular curves $c_1,c_2\in PC^1(I,\mathbb R^d)$ the infimum in \eqref{quotient-distance1} is attained by the constant speed reparametrizations of $c_1$ and $c_2$, i.e.:
$$\operatorname{dist}^{\mathcal S}_{a,b}([c_1],[c_2])=\operatorname{dist}_{a,b}(c_1\circ\psi_{c_1}^{-1}\circ \psi_{c_2},c_2)$$
where $\psi_c(u)=\int_0^u|\dot{c}(s)|\dd s.$ 
Indeed, let us first assume that $c_1$ and $c_2$ are both constant speed parametrized, i.e. $|\dot{c}_1|$ and $|\dot{c}_2|$ are constant on $I$ and equal to the curve lengths $\ell_{c_1}$ and $\ell_{c_2}$, respectively. Therefore $f_{a,b}$ is just constant, and the optimization problem becomes:
\begin{align}\label{quotien_distance_BV_a=0}
\operatorname{dist}^{\mathcal S}_{a,b}([c_1],[c_2])&=2b
\sqrt{\ell_{c_1}+\ell_{c_2}-2f_{a,b}\left(0,0\right)\sup_{\phi\in \mathcal D^*}\int_I\sqrt{\dot\phi(x)}  d x}.
\end{align}
We have by Cauchy-Schwarz
$$
\left(\int_I\sqrt{\dot\phi(x)}  d x\right)^2\leq 1\cdot\int_I\dot\phi(x) d x \leq 1\cdot1=1
$$
and for $\phi=\id$, the previous inequality is an equality. The result follows and we also obtain that the (pseudo-)distance is given by:
$$
\operatorname{dist}^{\mathcal S}_{a,b}([c_1],[c_2])=
\sqrt{\ell_{c_1}+\ell_{c_2}-2\sqrt{\ell_{c_1}\ell_{c_2}}} = \left|\sqrt{\ell_{c_1}} - \sqrt{\ell_{c_2}}\right|
$$
In the general case where $c_1,c_2\in PC^1(I,\mathbb R^d)$, one has that $c_1\circ\psi_{c_1}^{-1}$ and $c_2\circ\psi_{c_2}^{-1}$ have constant speed, and 
$$
\operatorname{dist}^{\mathcal S}_{a,b}([c_1],[c_2])  =\operatorname{dist}_{a,b}(c_1\circ\psi_{c_1}^{-1}\circ \psi_{c_2},c_2). 
$$
\end{remark}

\subsection{Existence of optimal reparametrizations on the space of closed curves}
We will now extend the previous result to the case of closed curves, i.e. when $D=S^1$. For closed  parametrized curves, there does not exist anymore an explicit formula for the geodesic distance associated to general elastic $G^{a,b}$-metrics---to our knowledge, the only $G^{a,b}$ metric with explicit geodesics for closed curves is the $a=b$ case for plane curves~\cite{younes2008metric} and space curves~\cite{needham2019shape}, both of which rely on specific constructions involving Hopf maps. Nevertheless, the formula we obtained for open curves in Theorem~\ref{genSRVT} still defines a reparametrization invariant distance function on the space of closed, parametrized curves: the next result follows by the same analysis applied in the open curve setting.

\begin{cor}
For $c_1,c_2\in AC_0(S^1,\mathbb R^d)$ let 
$\operatorname{dist}_{a,b}(c_1,c_2)$ be given by the same formula as in~\eqref{eq:dist} with integration over $I$ replaced by integration over $S^1$. Then  $\operatorname{dist}_{a,b}$
defines a metric on  the space of closed, absolutely continuous curves $AC_0(S^1,\mathbb R^d)$. 
\end{cor}

As previously, this allows us to construct a distance on the quotient space
of unparametrized closed curves by defining
\begin{equation}
\operatorname{dist}^{\mathcal S,\operatorname{cl}}_{a,b}([c_1],[c_2]) :=\inf_{\gamma\in \Gamma_{\operatorname{cl}}} \operatorname{dist}_{a,b}(c_1,c_2\circ\gamma),
\end{equation}
where $\operatorname{dist}_{a,b}$ is given by \eqref{eq:dist} and where $\Gamma_{\operatorname{cl}}$ denotes the group of absolutely continuous reparametrizations on the circle, i.e.,
\begin{align*}
\Gamma_{\operatorname{cl}}=\{\gamma \in AC(S^1,S^1):  \gamma \text{ is bijective and }\gamma'>0\quad a.e. \} 
\end{align*}
We introduce the shift operator on $S^1$ :
\begin{align*}
    S_{\tau} : \left\{\begin{array}{ccc}
        S^1 & \rightarrow & S^1  \\
        \theta & \mapsto & \theta + \tau 
    \end{array}\right.
\end{align*}
where $\tau \in S^1$.
Then we can rewrite the group of absolutely continous reparametrizations on the circle as 
\begin{align*}
\Gamma_{\operatorname{cl}}=\{S_{\tau}\circ\gamma: \gamma \in \Gamma \text{ and } \tau \in S^1 \}, 
\end{align*}
and we also need, as before, the closure of this group $\bar{\Gamma}_{\operatorname{cl}}=\{S_{\tau}\circ\gamma: \gamma \in \bar\Gamma \text{ and } \tau \in S^1 \}$. Then, the induced distance $\operatorname{dist}^{\mathcal S,\operatorname{cl}}_{a,b}$ can be expressed as

\begin{align}
\operatorname{dist}^{\mathcal S,\operatorname{cl}}_{a,b}([c_1],[c_2])&=\inf_{\gamma\in \Gamma_{\operatorname{cl}}}\operatorname{dist}_{a,b}(c_1,c_2\circ \gamma)=
\inf_{\gamma_1,\gamma_2\in \Gamma_{\operatorname{cl}}}\operatorname{dist}_{a,b}(c_1\circ\gamma_1,c_2\circ \gamma_2) \nonumber \\
&=\inf_{\gamma_1,\gamma_2\in \bar \Gamma_{\operatorname{cl}}}\operatorname{dist}_{a,b}(c_1\circ\gamma_1,c_2\circ \gamma_2). \label{quotient-distance-cl}
\end{align}

We can now formulate the existence result of optimal reparametrizations for closed curves. Our proof, which will use the same method as~\cite{hartman2021supervised} where  the result was shown for the SRV metric, is postponed to \ref{appendix:closed}.  
\begin{theo}\label{thm:existenceclosed}
Let $c_1,c_2\in PC^1(S^1,\mathbb R^d)$.
We have the following statements concerning the existence of optimal reparametrizations:
\begin{enumerate}
    \item $\mathbf{a<b:}$ for any $c_1,c_2\in PC^1(S^1,\mathbb R^d)$ there exists a strictly increasing homeomorphism  such that the infimum in~\eqref{quotient-distance-cl} is attained. If the derivatives $\dot c_1, \dot c_2$ are Lipschitz continuous then $\gamma \in \operatorname{Diff}_{C^1}(S^1)$, i.e., the optimal reparametrization is a $C^1$-diffeomorphism.. 
    \item $\mathbf{a\geq b:}$ for  $c_1,c_2\in PC^1(S^1,\mathbb R^d)$ there exists a pair of generalized reparametrization $\gamma_1,\gamma_2\in \bar \Gamma_{\operatorname{cl}}$ such that the infimum in~\eqref{quotient-distance-cl} is attained.
\end{enumerate}
\end{theo}

\section{Algorithms for the computation of quotient distances and geodesics}\label{sec:algorithms}
In the following, we will describe two different algorithms for the numerical computation of the optimal reparametrization on the space of open curves: an exact algorithm based on the work of Lahiri, Robinson and Klassen~\cite{lahiri2015precise} and a faster dynamic programming based approximation. Solving the registration problem on the space of closed curves simply requires an additional optimization over the starting point, i.e., one has to solve the registration problem on the space of open curves for any choice of starting point. Consequently, the numerical solution on the space of closed curves is significantly more expensive.

\subsection{Exact algorithm for piecewise linear curves}
 By the results of the previous section, we obtain the existence of optimal reparametrizations in the case of open, piecewise linear curves. Furthermore, using a result of Lahiri et. al.~\cite{lahiri2015precise} we obtain an explicit algorithm for the optimal reparametrizations $\gamma_1$ and $\gamma_2$. We have the following result that follows directly from the corresponding analysis for the SRV-metric.
\begin{theo}\label{thm:explicit_algorithm}
Let $c_1,c_2$ be two piecewise linear curves with values in $\mathbb R^d$. Then the pair of generalized reparametrizations $\gamma_1,\gamma_2\in \bar \Gamma$  that attains the infimum in~\eqref{quotient-distance2} consists of two piecewise linear maps. 
\end{theo}
\begin{proof}
First we note that piecewise linear curves, are piecewise smooth and thus in particular piecewise $C^1$.
This guarantees the existence of optimal reparametrization by the results of Theorem~\ref{thm:existence}.

Next, we introduce a notion from ~\cite{lahiri2015precise} and let $f: I\times I \to \R$. We call $f$ \emph{rectangular} if there exist partitions $0=i_0<i_1<\dots<i_m=1$ and $0=j_0<i_1<\dots<j_n=1$ such that $f$ is constant on each rectangle of the form $[i_{r-1},i_r]\times [j_{s-1},j_s]$. 

Using again Theorem~\ref{thm:existence} we have shown that finding optimal reparametrizations for the geodesic distance is equivalent to the optimization problem
\begin{equation}
\sup_{\gamma_1,\gamma_2\in \bar \Gamma} \int_I \sqrt{\dot{\gamma}_1(u)\dot{\gamma}_2(u)} f_{a,b}\left(\gamma_1(u),\gamma_2(u)\right) d u, 
\end{equation} 
where $f_{a,b}$ is defined in~\eqref{eq:f_ab}.
It is clear that the function $f_{a,b}$ is rectangular if the curves $c_1$ and $c_2$ are piecewise linear. From here, the proof given in~\cite{lahiri2015precise} goes through verbatim. 
\end{proof}

Consequently, the exact algorithm from \cite{lahiri2015precise} can be adapted to find optimal reparametrizations in our setting. This algorithm can find the optimal piecewise linear trajectory $\gamma_1,\gamma_2$ that maximizes $\int_{I \times I}W(\gamma_1(u),\gamma_2(u))\sqrt{\dot{\gamma}_1(u)\dot{\gamma}_2(u)}$ for general $W \in L^2(I\times I,\mathbb{R})$. For the $G^{a,b}$-metrics, we simply have $W=f_{a,b}$. We use the implementation from Martins Bruveris\footnote{https://github.com/martinsbruveris/libsrvf} that computes reparametrization for the usual $G^{1,1/2}$-metric, i.e. with $W(x,y)=q_1(x)\cdot q_2(y)$. Recall that closed curves require an additional optimization step to determine the starting point. For each vertex of the piecewise linear curve $c_1$, we simply compute the resulting distance with $c_2$ and choose the minimum.

\subsection{Dynamic programming approach}\label{sec:dynamic_programming}
The previous algorithm gives the exact optimal piecewise linear reparametrization, but is in practice very slow to compute. An alternative method is to use a dynamic programming scheme to estimate an approximation of the optimal reparametrization, similar to the approach proposed in \cite{trouve2000diffeomorphic,mio2007shape}.
\par
We define a discretization $\mathcal{I}=\{x_0,...,x_n\}$ of the intervall $[0,1]$ and restrict the search to piecewise linear functions on the grid $\mathcal{I}\times\mathcal{I}$. We denote by $PL_{\mathcal{I}}$ the set of piecewise linear increasing functions with vertices on $\mathcal{I}\times\mathcal{I}$. For $k<i$, $l<j$ and $\phi\in PL_{\mathcal{I}}$, we define the partial cost function
$$
E(k,l,\phi) = \int_{x_k}^{x_i}\sqrt{\dot{\phi}(x)} f_{a,b}(x,\phi(x)) d x.
$$
Let $L_{k,l,i,j}(x) = x_l + \frac{x_j-x_l}{x_i-x_k}(x-x_k)$ be the line between $(x_k,x_l)$ and $(x_i,x_j)$, and with a slight abuse of notation, we shall write $E(k,l,i,j)=E(k,l,L_{k,l,i,j})$ as the energy of the segment. By additivity, if $\phi\in PL_{\mathcal{I}}$ is defined by the segments $L_{k_1,l_1,k_2,l_2},L_{k_2,l_2,k_3,l_3},...,L_{k_{p-1},l_{p-1},k_p,l_p}$, the total energy of $\phi$ is
$$
E(\phi)=\sum_{n=1}^{p-1} E(k_n,l_n,k_{n+1},l_{n+1}).
$$
We thus have to find the sequence of nodes $(k_n,l_n)$ that minimizes the energy.
We define the partial value function to reach node $(i,j)$ by
\begin{equation}
   H(i,j) = \min_{k<i,l<j} E(k,l,i,j) + H(k,l)
\label{equH} 
\end{equation}
with $H(0,0)=0$; in other words, $H$ is defined recursively. Due to the specific additive form of $E$, if for all $k<i,l<j$, $H(k,l)$ is the minimal energy between $(0,0)$ and $(k,l)$, then by definition $H(i,j)$ is the minimal energy between $(0,0)$ and $(i,j)$. Consequently, the global minimal energy that we aim to find is given by $H(1,1)$.
\par
The algorithm proceeds in two steps. First, we compute the different values of $H$ on the grid $\mathcal{I}\times\mathcal{I}$ using equation (\ref{equH}). Then, we determine the optimal path by backtracking from vertex $(1,1)$: if $(x_i,x_j)$ is a vertex of the optimizer, we compute the node that joins $(x_i,x_j)$ by solving the problem:
$$
(\hat{k},\hat{l})=\underset{k<i,l<j}{\operatorname{argmin}} E(k,l,i,j) + H(k,l).
$$
To speed-up the computation of the value function and optimal reparametrization, a standard approach \cite{mio2007shape} is to restrict the search of the node that connects to $(i,j)$ to a smaller set than $\{k,l : k<i,l<j\}$. In our case, we define 
$$
N_{i,j}=\{k,l: i-6\leq k<i,j-6\leq l<j\}
$$
and the corresponding partial energy
$$
H(i,j)=\min_{(k,l)\in N_{i,j}} E(k,l,i,j) + H(k,l).
$$
This will lead to a restriction of the admissible slopes and in general to a less precise approximation of the reparametrization. Nevertheless, in all of our experiments, this restriction still yields good approximations of the true reparametrization functions, c.f. Figure~\ref{fig:optimal_reparameterizations}.

\subsection{Examples}\label{sec:examples}

Figure~\ref{fig:two_open_geodesics} shows several geodesics between pairs of plane curves---a pair of simple synthetic curves and a pair of real leaf shapes. Geodesics are computed for a variety of elastic metrics $G^{a,b}$; we fix $b = \frac{1}{2}$ and compute geodesics for $a \in \{0.1,0.5,1,5\}$. All geodesics in this figure were computed using our dynamic programming algorithm. These first examples clearly illustrate that the intermediate shapes along the geodesics strongly depend on the choice of metric parameters. Figure~\ref{fig:optimal_reparameterizations} compares the estimated reparametrization functions found via the dynamic programming algorithm to those found by the exact algorithm.  We see here that the dynamic programming algorithm typically returns registrations which are close to the true optimal ones, while incurring a lower numerical burden. Indeed, for the synthetic example, the average computational times for exact registrations were 306s, 269s, 40s, and 0.2s, respectively, for $a = 0.1, 0.5, 1, 5$; their counterparts computed using the dynamic programming algorithm were orders of magnitude smaller: 0.23s, 0.09s, 0.07s, and 0.06s. A similar trend held for the leaf shapes, where we got 299s, 253s, 34s, and 0.19s for the exact algorithm and 0.22s, 0.03s, 0.025s, and 0.021s for dynamic programming. In the figure, we report the performance of the dynamic programming algorithm by giving its relative error $(d_{\mathrm{dyn}} - d_{\mathrm{ex}})/d_{\mathrm{ex}}$, where $d_{\mathrm{dyn}}$ is geodesic distance (for given metric parameters) computed via dynamic programming and $d_{\mathrm{ex}}$ is the exact distance.

\begin{figure}
\centering
\includegraphics[width = \textwidth]{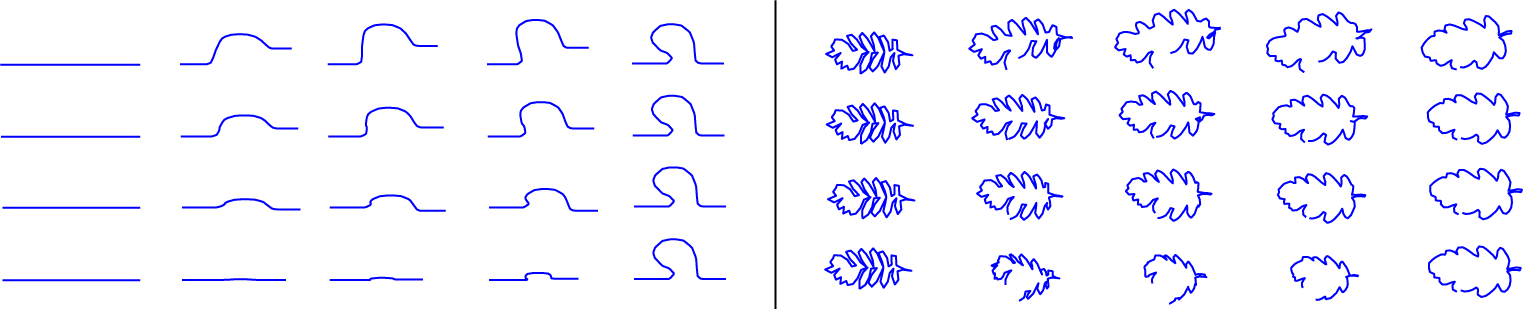}
\caption{Geodesics between open plane curves. Each panel shows several geodesics between a pair of curves, with each row corresponding to a different choice of parameters in the elastic metric. In the left panel, the source and target curves are simple synthetic curves, designed to clearly illustrate the dependence of the geodesic  on the metric parameters. In the right panel, the source and target curves are real leaf shapes---although the source and target curves are actually closed, we compute each geodesic in the space of open curves. In both the left and right panels, the metric parameters vary by row and are given by $G^{a,\frac{1}{2}}$, with $a = 0.1$,  $0.5$, $1$ and $5$, respectively.}
\label{fig:two_open_geodesics}
\end{figure}

\begin{figure}
\centering
\includegraphics[width = 0.99\textwidth]{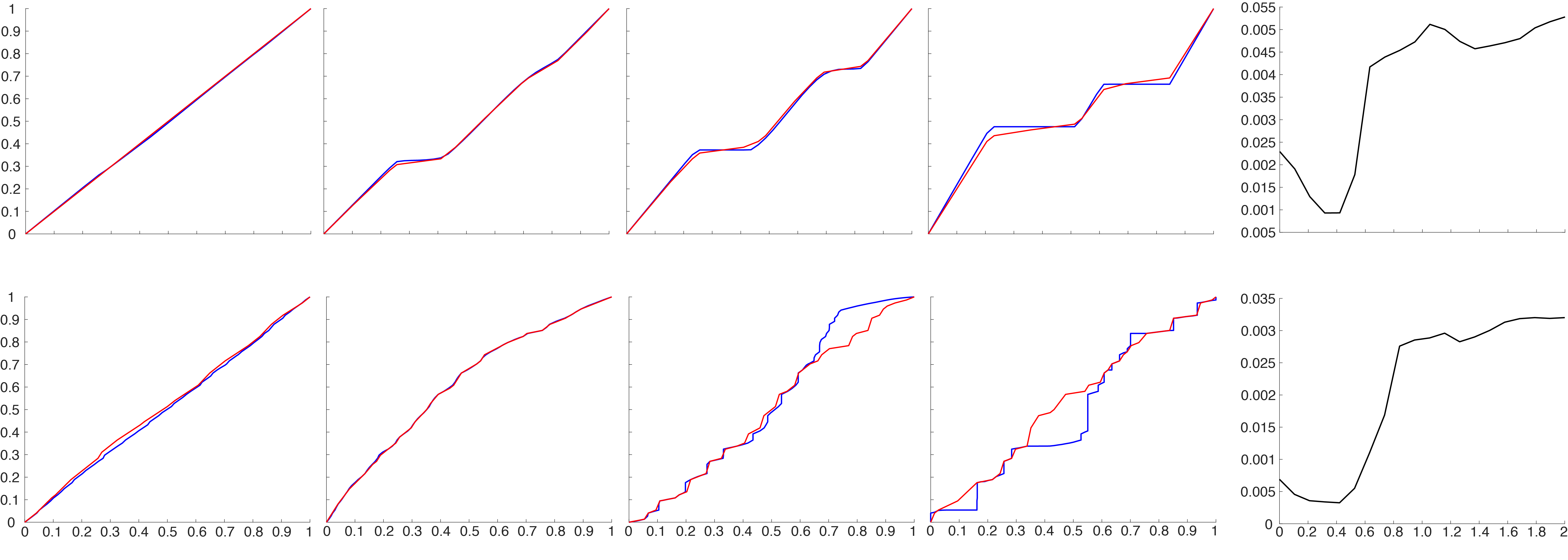}
\caption{Comparison of registrations via the exact and dynamic programming algorithms. The top row shows the registrations (i.e., optimal diffeomorphisms of $[0,1]$) between the synthetic curves from Figure~\ref{fig:two_open_geodesics} with respect to metric parameters $b=\frac{1}{2}$ and $a = 0.1,0.5,1,5$, respectively. In each figure, the exact registration is plotted in blue and the dynamic programming registration is plotted in red. The last graph in the row plots the relative error of dynamic programming versus the exact algorithm (see the text) for a finer range of parameters $a \in \{0, 0.1, 0.2, \ldots, 2\}$, $b = \frac{1}{2}$. The bottom row shows the same experiment for the registrations between the leaf shapes of Figure~\ref{fig:two_open_geodesics}. Observe that the relative error is generally on the order of a single digit percentage in each case.}
\label{fig:optimal_reparameterizations}
\end{figure}

Figure~\ref{fig:protein_geodesics} applies our framework to compute geodesics between 3D curves (protein backbones), for a variety of elastic metric parameters. These geodesics were computed, once again, using the dynamic programming algorithm. 

\begin{figure}
\centering
\includegraphics[width = 0.95\textwidth]{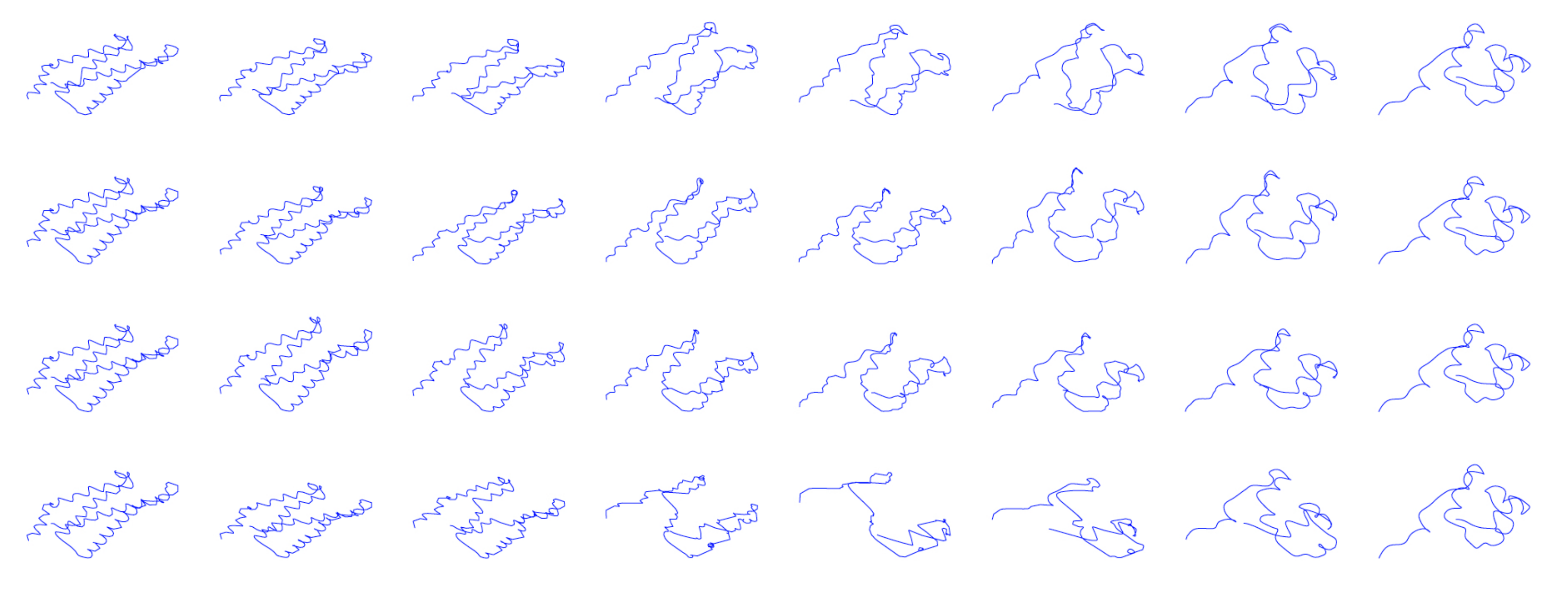}
\caption{Geodesics between space curves. Each row shows a geodesic between the same pair of protein backbones, modeled as space curves, for a particular choice of elastic metric. The elastic metric parameters for the rows are $b = \frac{1}{2}$, $a = 0.1,\,0.5,\,1,\,5$, respectively.}
\label{fig:protein_geodesics}
\end{figure}

\section{Metric learning}\label{sec:metric_learning}

The numerical experiments of the previous section illustrate both the influence and importance of the choice of metric when comparing and matching shapes. While certain heuristics may sometimes guide the selection of the parameters $a$ and $b$ of the elastic metric for a given dataset and application, it is often done via empirical trial and error approaches. Thus, in recent years, there has been a growing interest in developing methods for automatically estimating metrics on shape spaces. 

Obviously, there is a priori no natural criterion to prefer one metric over another for the basic task of matching two shapes. However, when considering shape datasets and problems such as clustering or classification, it should be expected that different metrics will lead to different ways of quantifying differences across samples and consequently different properties from a statistical perspective. This suggests the idea of attempting to optimize (or learn) the choice of metric in order to improve the statistical power of shape analysis methods. The elastic framework of this paper appears quite amenable to such a task as learning the metric here reduces to learning a single parameter (the ratio of $a$ and $b$). In this section, we consider the issue of metric learning for shape classification using two different models. Our primary focus is on demonstrating the feasibility and advantages of optimizing the choice of the metric in such a context; we leave an in depth study of the issue of developing efficient numerical algorithms for metric learning for future work.      

The literature on metric learning is vast, and we only summarize some main ideas here---for more details, there are several surveys on the topic, e.g., \cite{yang2006distance,bellet2013survey,kulis2013metric,kaya2019deep}. Generally, metric learning is a supervised machine learning technique for choosing a metric from a parametric family which optimally separates data coming from different classes. Classically, the data consists of Euclidean feature vectors, and the metrics under consideration are \emph{Mahalanobis distances}~\cite{xing2002distance,davis2007information,ying2012distance}, which allows training via standard techniques from convex optimization theory. Closer to the topic of this paper, there has also been recent interest in metric learning on parametric families of Riemannian metrics on manifolds, such as spaces of SPD matrices~\cite{vemulapalli2015riemannian}, spaces of histograms~\cite{le2015unsupervised} and graphs~\cite{heitz2021ground}.

A common paradigm in metric learning is to represent data via \emph{pairwise constraints}, where the training data consists of two sets $S = \{(x_i,x_j)\}$ and $D = \{(x_k,x_\ell)\}$ so that each pair $(x_i,x_j) \in S$ consists of \emph{similar points} (coming from the same class) and each pair $(x_k,x_\ell) \in D$ consists of \emph{dissimilar points} (coming from different classes). One then designs a loss function on the metric parameter space which encourages distances between similar points to be small and distances between dissimilar points to be large---the particulars of the loss function are application-dependent, and several choices are described in the survey papers cited above. An optimal metric with respect to a given loss can then be used for downstream distance-based analysis tasks, such as clustering, dimension reduction and $k$-nearest neighbors classification. This is the approach that we take in Section \ref{ssec:pairwise_constraints}, where we us the pairwise constraints method to train metrics for various 2-dimensional shape datasets. On the other hand, if one has a particular classification task in mind then it is sensible to learn the metric which optimizes performance on this task directly. We take this approach in Section \ref{ssec:shape_classification} to learn a metric which optimally classifies 3-dimensional protein backbone curves. 

\subsection{Pairwise Constraints}
\label{ssec:pairwise_constraints}
Let us first consider the goal of estimating, in a supervised fashion, the metric that will best separate 2-dimensional shapes according to the pairwise constraints paradigm described above. In other words, suppose that we have training data consisting of a collection of unparametrized curves $X = \{c_j\}_{j=1}^N$ together with known labels $y = \{y_j\}_{j=1}^N$ with each $y_j \in \{1,\ldots,K\}$ (that is, there are $K$ distinct classes). We seek to determine the optimal parameters $(a,b)$ for the elastic metric $G^{a,b}$ so that the geodesic distance $\operatorname{dist}_{a,b}^{\mathcal S}$ optimally separates those classes. By a simple normalization argument, we may fix one of the two metric parameters, which we shall do in the following by setting $b=1/2$ and optimize over $a\geq 0$.  

The above task can be stated formally by introducing an adequate pairwise constraint loss function depending on the distances $\operatorname{dist}^\mathcal{S}_{a,b}(c_i,c_j)$ between the curves of the training set. There have been various families of such loss functions appearing most notably in the machine learning literature. As a proof-of-concept, we choose here a simple loss which we can loosely think of as the ratio of the intra-class variance by the inter-class variance of the shape distances. Specifically, let $S \subset X \times X$ be the collection of ordered pairs $(c_i,c_j)$ of curves with the same label ($y_i=y_j$) and $D \subset X \times X$ the collection of ordered pairs with different labels ($y_i\neq y_j$). We define our loss function by:
\begin{equation}\label{eqn:loss_function_clustering}
L(a) =  \left(\frac{1}{|S|}\sum_{(c_i,c_j) \in S} \operatorname{dist}^\mathcal{S}_{a,1/2}(c_i,c_j)^{2}\right)^{\frac12} \cdot \left( \frac{1}{|D|} \sum_{(c_k,c_\ell) \in D} \operatorname{dist}^\mathcal{S}_{a,1/2}(c_k,c_\ell)^{2}\right)^{-\frac12}.
\end{equation}
 Minimizing $L(a)$ should be achieved when one strikes a balance between concentrating curves from the same class close to one another, while giving a large separation between different classes. As was recently observed in \cite{ghojogh2019fisher, ghojogh2022spectral}, minimizing this loss function is essentially equivalent to solving the metric learning optimization problem described in the pioneering work of Xing et al. \cite{xing2002distance}.

For a given value of $a$, this loss function is calculated by first computing each of the $N(N-1)/2$ pairwise distances $\{\operatorname{dist}^\mathcal{S}_{a,1/2}(c_i,c_j)\}$, which requires solving each of the corresponding registration problems. In our experiments, this is done using the dynamic programming scheme for the sake of computational efficiency. For the purpose of this work, we evaluate the loss function over a range of different values for $a$ in order to determine the approximate value of the minimizer. Note that, while the expression of the loss when reparametrizations are fixed is relatively simple and could be optimized easily with respect to $a$, the difficulty is that the optimal reparametrizations leading to each pairwise distance value also depend on $a$. This makes the derivation of more sophisticated and efficient schemes for the minimization of \eqref{eqn:loss_function_clustering} a non-trivial problem that we leave to future investigation, c.f. the discussion below.   

\vskip2ex

\textbf{Results.} We tested our metric learning pipeline on two shape datasets. Results from the first dataset are reported in Figure~\ref{fig:leaf_metric_learning}. Here, the data consisted of leaf shapes coming from four classes with 20 samples in each class. We chose 7 random examples from each class as training data and created pairwise data $S$ and $D$, with $S$ consisting of all pairs from the same class and $D$ consisting of all pairs from different classes. We then minimized the loss function \eqref{eqn:loss_function_clustering} via a grid search over the one-dimensional parameter space (a gradient descent algorithm was also implemented, which gave the same results). The efficacy of the learned metric was validated by testing the ability of the resulting geodesic distance to separate shapes from different classes. We measured this by computing the pairwise distance matrices for both the training and testing data, partitioning each dataset into four classes by applying complete linkage hierarchical clustering to the distance matrices and computing the Rand index of the inferred clusters against the ground truth classes. Similar results for the second dataset, consisting of shapes of four different species of animals, are reported in Figure~\ref{fig:bug_metric_learning}. Notably, the optimal metrics computed for the two datasets are quite different, indicating that the choice of optimal metric is data-dependent. 

\begin{figure}
\centering
\includegraphics[width = 0.9\textwidth]{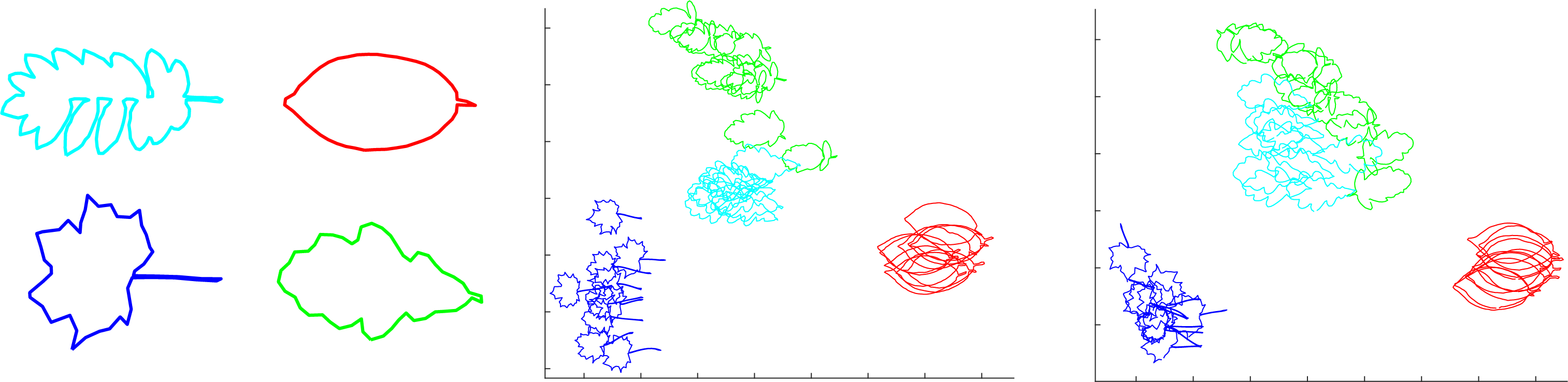}
\caption{Metric learning on the leaf dataset. The left panel shows representatives of each of the four classes in the leaf dataset. The middle and right panels show MDS (multidimensional scaling) plots of the pairwise distance matrices with respect to the learned metric for the training and testing data, respectively. The optimal metric parameter learned for the loss function \eqref{eqn:loss_function_clustering} is $a=0.71$. The Rand indices (see text for a full description) for the training and testing sets are 1 and 0.93, respectively.}
\label{fig:leaf_metric_learning}
\end{figure}

\begin{figure}
\centering
\includegraphics[width = 0.9\textwidth]{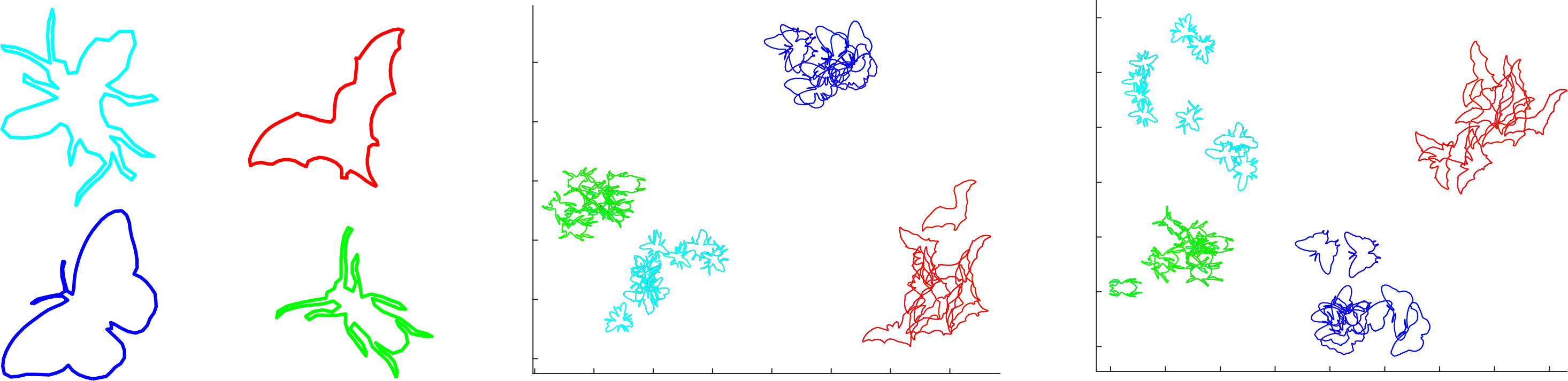}
\caption{Metric learning on the animal dataset. Similar to Figure~\ref{fig:bug_metric_learning}, from left to right the panels show representative shapes from each of the four classes, an MDS plot for the training data  and an MDS plot for the testing data, both with respect to the pairwise distance matrices computed with the optimal metric parameter. For this experiment, the optimal parameter was $a = 0.52$ and the Rand indices for the training and testing sets were both equal to $1$ (i.e., perfect clustering).}
\label{fig:bug_metric_learning}
\end{figure}

\subsection{Metric learning for shape classification}\label{ssec:shape_classification}
As an alternative to optimizing the metric parameters  with respect to the loss \eqref{eqn:loss_function_clustering}, one may instead consider trying to maximize classification scores on the training set in a cross-validation fashion. A simple approach could be, for example, to evaluate leave-one-out nearest neighbor classification scores for varying values of $a$. A usually more robust way however is to rather rely on the estimation of conditional probabilities for the different classes. We present one possible approach in the following description.  

With the same notation as in the previous section, we introduce a leave-one-out scheme in which for each $j=1,\ldots,N$, we denote by $p_j(\ell|\hat{X}_j)$ the probability of $c_j$ to be in the class $\ell \in\{1,\ldots,K\}$ knowing only the curves in $\hat{X}_j = \{c_j\}_{j\neq i}$ and their labels. In order to estimate the conditional probabilities, we adapt a standard approach in many machine learning works, see e.g. \cite{goodfellow2016deep}, Chapter 6.2. Namely, we set $p^{(j)}(\ell|\hat{X}_j) = \sigma(z^{(j)})_\ell$ where the vector $z^{(j)} \in \R^k$ is defined by
\begin{equation*}
    (z^{(j)})_\ell = -\frac{1}{|\{i\neq j: y_i=l\}|}\sum_{i\neq j, y_i=l} \operatorname{dist}^a(c_i,c_j)^2
\end{equation*}
and $\sigma$ is the softmax function given by
\begin{equation*}
    \sigma(z)_\ell = \frac{e^{\beta z_\ell}}{\sum_{m=1}^{K} e^{\beta z_\ell}}
\end{equation*}
with $\beta>0$ a fixed (or tunable) parameter. Then, for each $j=1,\ldots,N$, $p^{(j)}(y_j|\hat{X}_j)$ gives an estimate of the probability for the correct class $y_j$ of curve $j$. Thus, we seek to maximize the sum of the log-likelihoods over all instances of the leave-one-out scheme. In other words, we define the loss function to minimize to be:  
\begin{equation*}
    L(a) = -\sum_{j=1}^{N} \log(p^{(j)}(y_j|\hat{X}_j)).  
\end{equation*}
Note that $-\log(p^{(j)}(y_j|\hat{X}_j))$ can also be interpreted as the Kullback-Leibler divergence between the true probability distribution $\delta_{y_j}$ and the estimated one $p^{(j)}(\cdot|\hat{X}_j)$. As in the previous section, we can calculate the above loss for each value of the metric parameter $a$ by first solving all pairwise matching problems to obtain the set of distances $\{\operatorname{dist}^a(c_i,c_j)\}$.   

\vskip 2ex

\textbf{Results.}
We used data from the 3D Shape Retrieval Contest 2010 (SHREC’10) \cite{shrec} which contains a training dataset of 1000 protein structures from 100 classes, each class containing 10 proteins. Only the 3D curves representing the protein backbones were extracted and used in our experiments; two examples of such protein shapes were shown earlier in Figure \ref{fig:protein_geodesics}. In addition, 50 more proteins from random classes formed the testing dataset that we used to evaluate the metric learning process. Our goal is to compare results obtained with the above approach to the methods from~\cite{shrec}. We calculated the classification loss for a sample of values of the metric parameter $a$ between $0$ and $2$ (Figure \ref{fig:leave_one_out}) and found the optimal value to be $a=0.73$. 
\par

\begin{figure}
    \centering
    \begin{tabular}{cc}
         \includegraphics[width=0.35\textwidth]{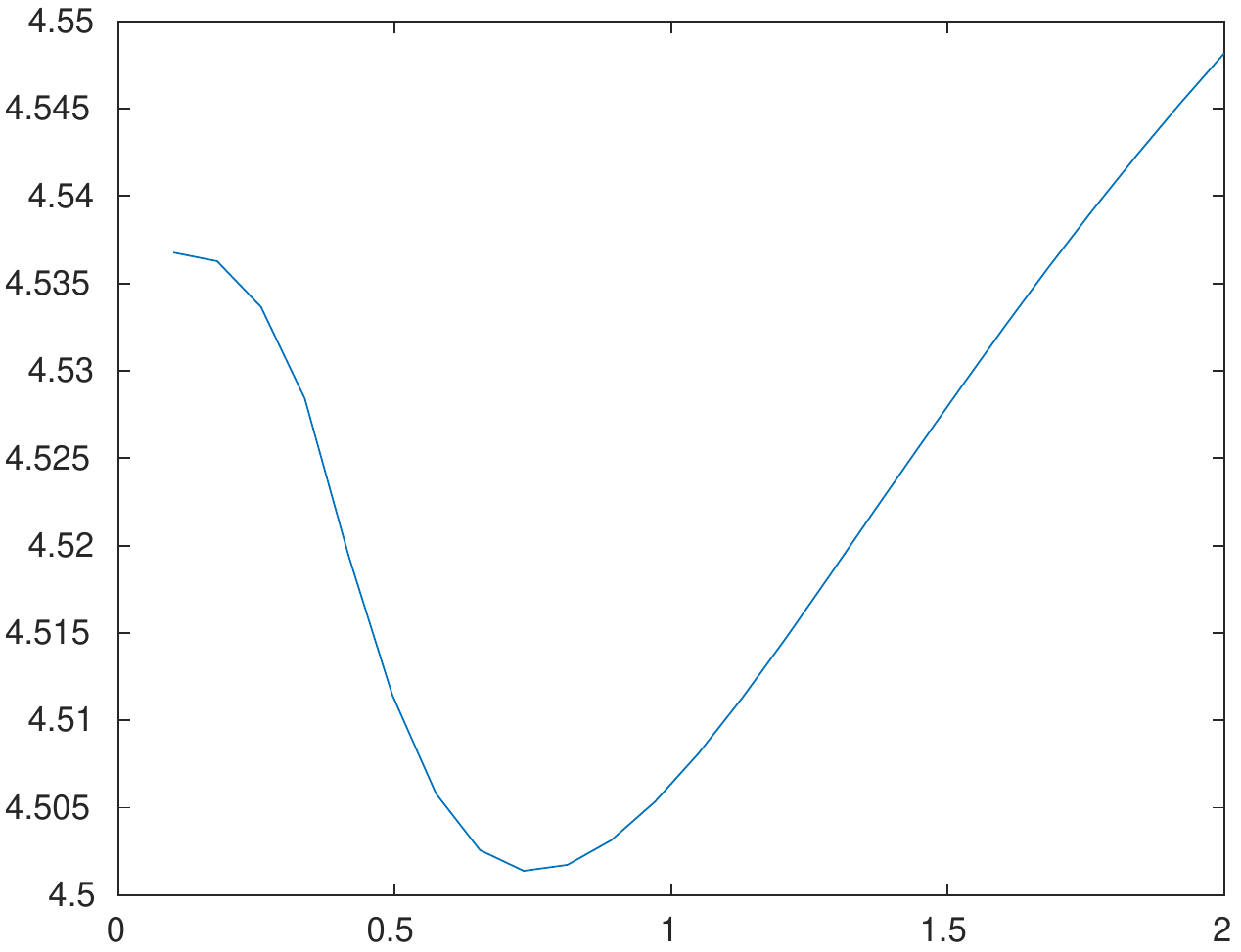}&  
         \includegraphics[width=0.35\textwidth]{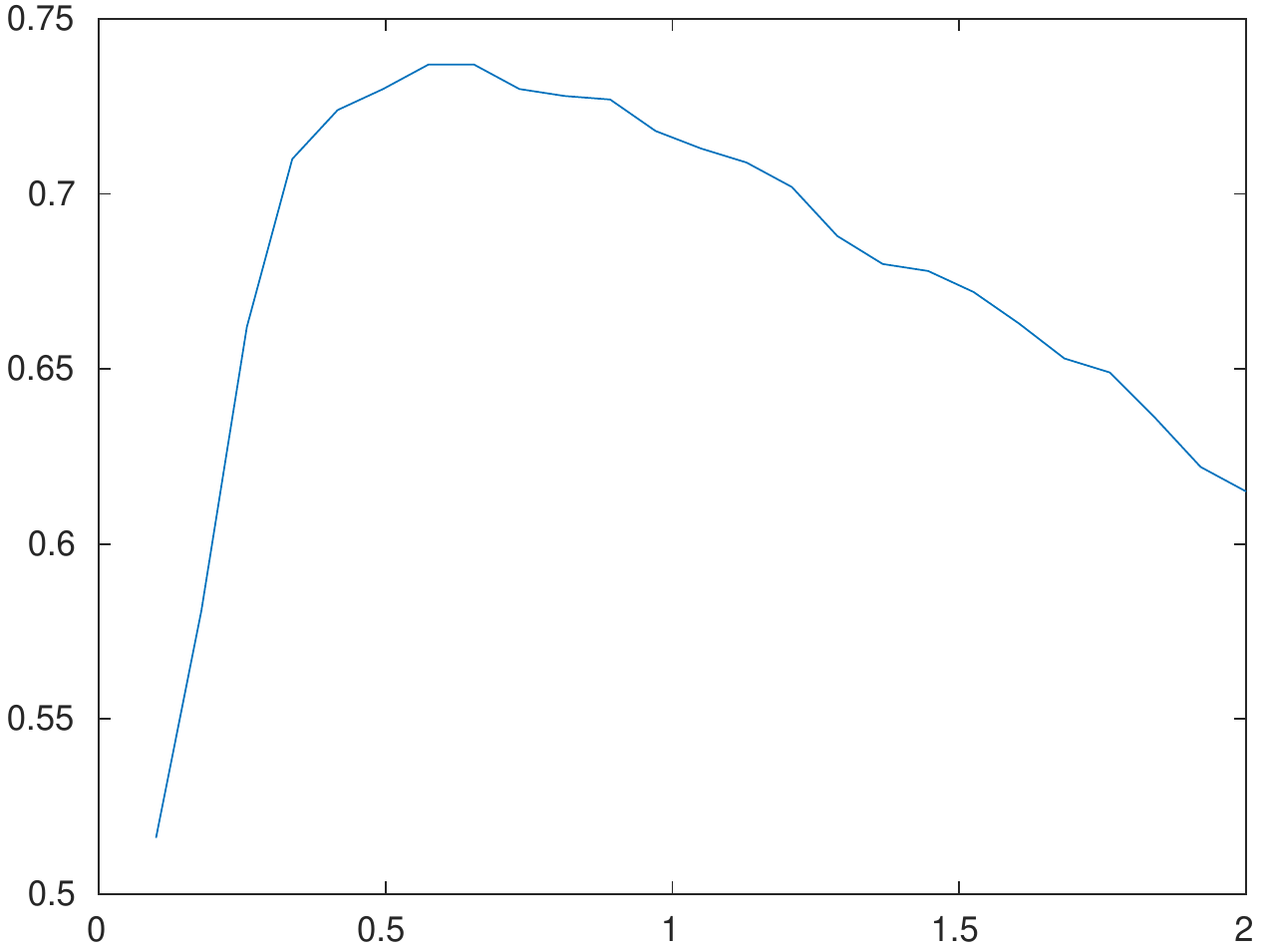}
    \end{tabular}
    \caption{Scores for different values of $a$. In the left panel we show the loss function described in Section \ref{ssec:shape_classification}. For comparison, on the right, we also evaluate the cross-validation leave-one-out accuracy based on nearest neighbor classification.}
    \label{fig:leave_one_out}
\end{figure}

For evaluation, we use this optimal parameter to compute a matrix of the distances between the 50 proteins in the testing set and each of the proteins in the training set. The performance of the method is measured in the same two ways as in the original contest.
\begin{itemize}
    \item Nearest neighbor: for each of the 50 proteins, we find the closest protein from the training dataset to predict the class of the testing protein. We calculate the overall percentage of correct predictions. For the optimal value of the metric parameter determined by our method -- $a=0.73$ -- we obtain $82\%$ of correct predictions, which beats all methods from \cite{shrec} (the best method in the paper reaches $80\%$). 
    
    \item Receiver operating characteristic (ROC) curve: for each of the 50 proteins in the testing dataset, we create a ranked list of the proteins from the training dataset, from the closest protein to the most distant. This ranked list contains 10 proteins in the actual class of the testing protein (the true positives) and 990 proteins in a different class (the true negatives). The ranked list is traversed sequentially and we plot the cumulative rate of true positives against the cumulative rate of true negatives. Figure \ref{fig:Roc_curves} shows aggregate ROC curves for all of the 50 test protein curves, for different parameters of the metric. For the optimal value of $a$, we also compute the area under the curve (AUC) for the ROC curves of each testing protein and plot it in the right panel of the figure.
\end{itemize}

\begin{figure}[H]
    \centering       \includegraphics[width=0.4\textwidth]{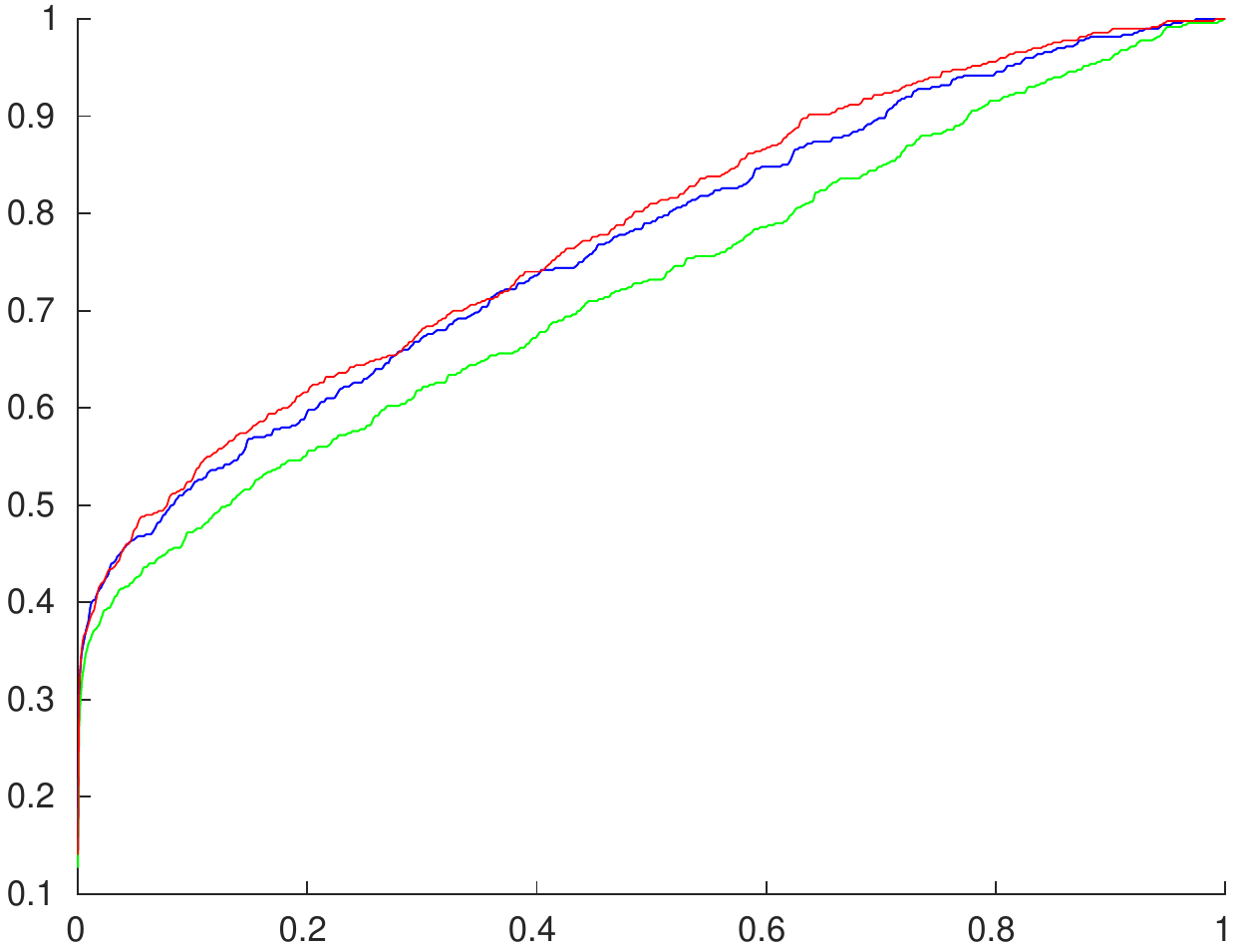}\qquad
      \includegraphics[width=0.4\textwidth]{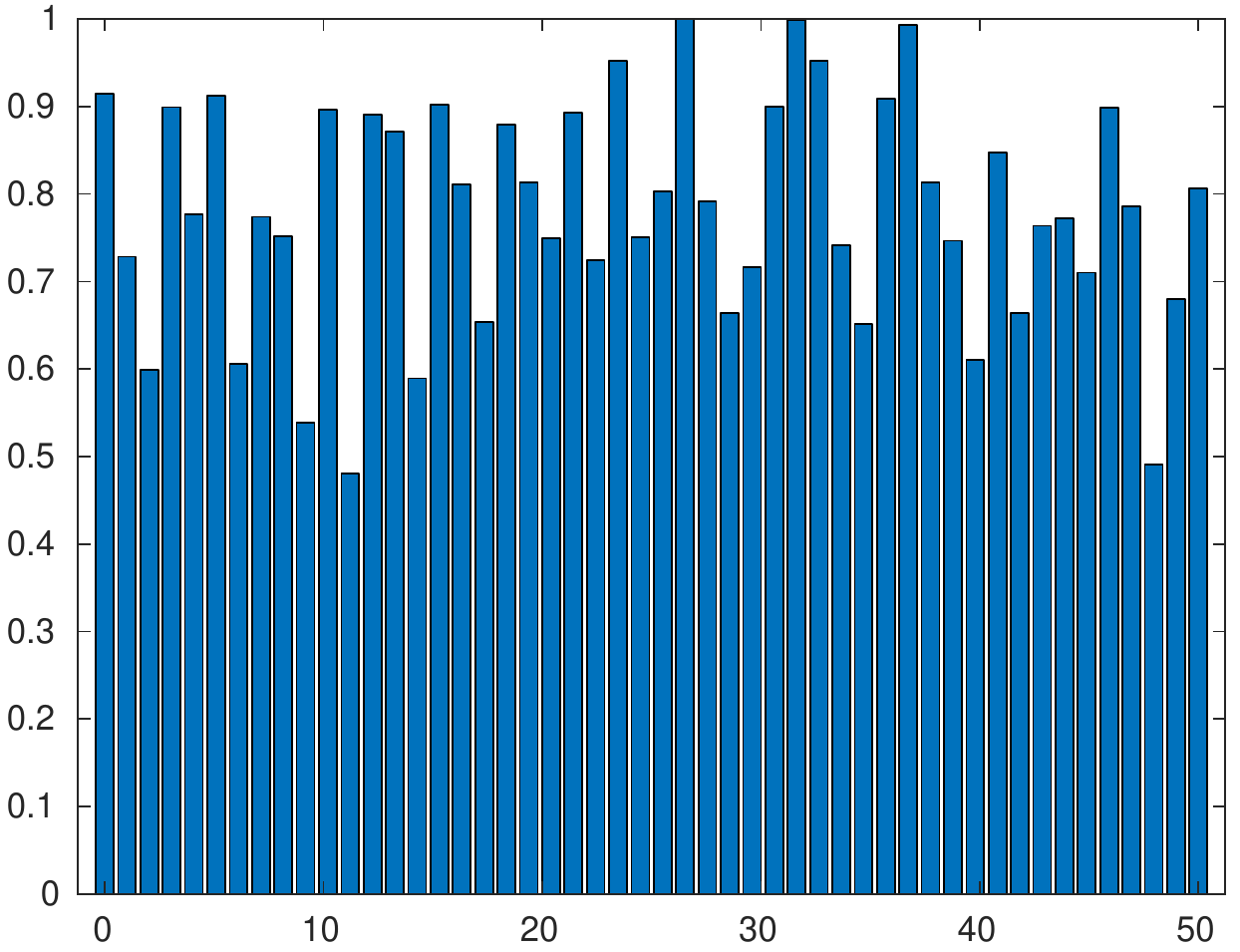}
        \caption{Left: Aggregate ROC curves for different metric parameters: $a=2$ (green), $a=1$ (blue), and the optimal value $a=0.73$ (red).
        Right: Bar chart showing the AUC for optimal metric parameter value $a=0.73$ for each testing protein.
        }
     \label{fig:Roc_curves}
\end{figure}

\bibliographystyle{abbrv}
\bibliography{ref}

\appendix

\setcounter{lemma}{0}
\renewcommand{\thelemma}{\Alph{section}.\arabic{lemma}}
\setcounter{remark}{0}
\renewcommand{\theremark}{\Alph{section}.\arabic{remark}}
\setcounter{theo}{0}
\renewcommand{\thetheo}{\Alph{section}.\arabic{theo}}
\setcounter{figure}{0}
\renewcommand{\thefigure}{\Alph{section}.\arabic{figure}}

\section{An interpretation from linear elasticity theory}
\label{subsec:elasticity}
With the definitions and notations of Section \ref{ssec:elastic_metrics}, let us consider a portion of an open smooth curve which is parametrized on an interval we will denote $I$, i.e. $c:I\rightarrow \R^2$ is an immersion of class $C^\infty$, and define the thin shell domain $\Omega \subset \R^2$ of ``thickness'' $\delta>0$ around that curve. Specifically, we take $\Omega$ as being given by the parametrization $\omega: [-\delta/2,\delta/2] \times I\rightarrow \Omega$ defined by $\omega(t,u) = c(u) + t n(u) $ where $n(u)$ is the unit normal vector to the curve $c$ at $u$. It is easy to see that for $\delta$ small enough, $\omega$ is a diffeomorphism that we can view as a foliation of $\Omega$. Indeed, we note that $\omega(0,\cdot) = c(\cdot)$ and for any $t \in [-\delta/2,\delta/2]$, $\omega_t: u\mapsto \omega(t,u)$ defines a parametrized curve which corresponds to layer $t$ of the foliation. Moreover, $\omega_t'(u) = c'(u) + t n'(u)$ and, as $n(u)$ is a unit vector orthogonal to $c'(u)$, we get that $\omega_t'(u)$ is parallel to $c'(u)$ which implies that $n(u)$ is also the unit normal vector to $\omega_t$ at $u$. Thus, for any $x=\omega(t,u) \in \Omega$, we can define the orthonormal vector frame $F(x) = (\tau(x),n(x))$ by:
\begin{equation*}
 \tau(x) = \frac{\omega_t'(u)}{|\omega_t'(u)|} = \frac{\partial_u\omega(t,u)}{|\partial_u\omega(t,u)|}, \ \ n(x) = n(u)
\end{equation*}
See the illustration given in Figure \ref{fig:2D_elastic_shell}. 

\begin{figure}
 \centering
 \includegraphics[width=10cm]{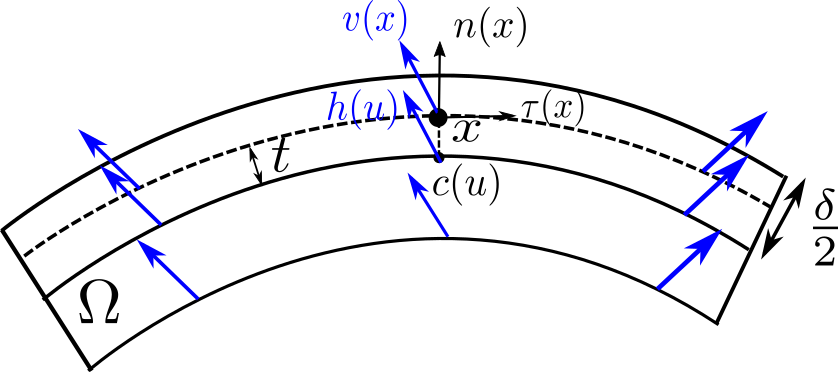} 
 \caption{Illustration of the thin shell elastic domain model.} \label{fig:2D_elastic_shell}
\end{figure}

Now, we model $\Omega$ as a linear elastic material which undergoes an infinitesimal deformation given by a smooth vector field $v:\Omega \rightarrow \R^2$. We shall further assume that this deformation field is uniform along the transversal direction, in other words that it takes the following form: for any $x=\omega(t,u)$, $v(x) = h(u)$ where $h$ is a vector field defined along the curve $c$ as in the previous section, c.f. again Figure \ref{fig:2D_elastic_shell} for visualization. Note that this is a natural assumption in the small thickness laminar model that we are interested in here. Then, following the approach of classical linear elasticity \cite{gurtin1981,ciarlet1988three-dimensional}, one introduces the $(2 \times 2)$ symmetric tensor field defined for all $x \in \Omega$ as $\varepsilon(x)=\frac{dv(x) + dv(x)^T}{2}$. This is known as the strain tensor associated to the deformation field $v$ and expressed in the canonical basis. Given the specific laminar structure of the domain here, it will be more convenient to instead consider the strain tensor relative to the above orthonormal frame $F(x)$, which is specifically $S(x) = F(x)^T \varepsilon(x) F(x)$.  The linear elastic energy associated to the deformation field is obtained from Hooke's law and take the general form:
\begin{equation}
 E(v) = \int_{\Omega} U(x,S(x)) dx
\end{equation}
where $U(x,\cdot)$ is a quadratic form on the space of symmetric $(2\times 2)$ matrices which is usually referred to as the elastic or stiffness tensor. In the present context, we will restrict the class of such elastic tensors by making a few additional assumptions. First, we will consider the elastic properties of the material to be uniform in the sense that $U(x,\cdot)$ does not depend on $x \in \Omega$. Then, viewing any symmetric matrix $S=(s_{ij})_{i,j=1,2}$ as the $(3\times 1)$ vector $(s_{11},s_{22},s_{12})^T$, the quadratic form $U$ can be identified with a single symmetric positive definite $(3 \times 3)$ matrix which we write as:
\begin{equation*}
 U = \begin{pmatrix} 
      u_{1,1} & u_{1,2} & u_{1,12} \\
      u_{1,2} & u_{2,2} & u_{2,12} \\
      u_{1,12} & u_{2,12} & u_{12,12}
     \end{pmatrix}.
\end{equation*}
The coefficients in $U$ assign weights to the different terms in the elastic energy in the following way. Both coefficients $u_{1,1}$ and $u_{2,2}$ correspond to spring-like stiffness coefficients in the tangential and normal directions respectively. Coefficient $u_{1,2}$, on the other hand, weighs the relative compression/stretching between tangential and normal direction. The coefficient $u_{12,12}$ can be associated with bending energy that results from a change of angle between the two directions. We will make some further symmetry assumptions on the material, namely that it is orthotropic with respect to the two directions $\tau(x)$ and $n(x)$ at each point $x$. This leads to the conditions $u_{1,12} = u_{2,12}=0$. Note that the orthotropy assumption is relatively common in many materials (with the exception of certain crystals) and include in particular fully isotropic materials, c.f. Remark \ref{rem:isotropic_materials} below. 

Going back more specifically to the deformation of the foliated domain $\Omega$, due to the particular form of the vector field $v$, we can see that for any $x=\omega(t,u) \in \Omega$, $dv(x) \cdot n(x) = 0$. This implies that the strain tensor is of the form:
\begin{equation*}
 S(x) = \begin{pmatrix}
         \tau(x)^T dv(x)\cdot \tau(x) & \frac{1}{2} (dv(x)\cdot \tau(x))^Tn(x) \\ 
         \frac{1}{2} (dv(x)\cdot \tau(x))^T n(x) & 0
        \end{pmatrix}
\end{equation*}
and so, identified as a $(3\times 1)$ vector, we have $S(x)=(s_{11}(x), 0 , s_{12}(x))^T$. Under the previous assumptions on the elastic tensor, we find the the following elastic energy:
\begin{align*}
 E(v) &=\int_{\Omega} (u_{1,1} s_{11}(x)^2 + u_{12,12} s_{12}(x)^2) dx \\
 &= \int_{-\frac{\delta}{2}}^{+\frac{\delta}{2}} \int_{I} (u_{1,1} s_{11}(\omega(t,u))^2 + u_{12,12} s_{12}(\omega(t,u))^2) \left|J_\omega(t,u) \right| du dt . 
\end{align*}
where $J_\omega$ denotes the Jacobian determinant of $\omega$.  We have seen that $\partial_t \omega(t,u) = n(u)$ and $\partial_t \omega(t,u) = c'(u) + t n'(u)$ with $c'(u)$ and $n'(u)$ being parallel vectors both orthogonal to $n(u)$. Thus, for all $(t,u)$, $|J_\omega(t,u)| = |c'(u) + t n'(u)|$.

Now, using the continuity with respect to $t$ of the inside integral and the mean value theorem, we obtain that:
\begin{align*}
 \lim_{\delta \rightarrow 0} \frac{1}{\delta} E(v) = \int_{I} (u_{1,1} s_{11}(\omega(0,u))^2 + u_{12,12} s_{12}(\omega(0,u))^2) |c'(u)| du.
\end{align*}
In addition, from $v(\omega(0,u)) = h(u)$, we get by differentiating that $dv(\omega(0,u))\cdot \tau(\omega(0,u)) = h(u)/|c'(u)|$. We can then rewrite the above expressions of $s_{11}$ and $s_{12}$ as: 
\begin{align*}
 &s_{11}(\omega(0,u)) = \tau(\omega(0,u))^T \left(\frac{h'(u)}{|c'(u)|} \right) =\left(\frac{c'(u)}{|c'(u)|}\right)^T \frac{h'(u)}{|c'(u)|}  \\
 &s_{12}(\omega(0,u)) = \frac{1}{2} \frac{n(\omega(0,u))^T h'(u)}{|c'(u)|}
\end{align*}
which finally leads to:
\begin{equation*}
 \lim_{\delta \rightarrow 0} \frac{1}{\delta} E(v) = \int_{I} \left[u_{1,1} (D_s h^{\top})^2 + \frac{u_{12,12}}{4} (D_s h^{\bot})^2\right] ds
\end{equation*}

In summary, we have shown that the expression of the Riemannian metric $G^{a,b}(h,h)$ of \eqref{eq:our_metric_ab} is obtained as the limit of the elastic energy of an orthotropic laminar thin shell domain as the thickness $\delta \rightarrow 0$, in which $a^2=u_{11}$ and $b^2=u_{12,12}/4$ can be interpreted as stretching and bending energy coefficients respectively.

\begin{remark}
\label{rem:isotropic_materials}
 In the special case of an isotropic elastic domain $\Omega$ (still with respect to the frame vectors $\tau(x)$ and $n(x)$), the elastic tensor takes the particular form:
 \begin{equation*}
  U = \begin{pmatrix} 
      2\mu+\lambda & \lambda & 0 \\
      \lambda & 2\mu+\lambda & 0 \\
      0 & 0 & 2\mu
     \end{pmatrix}.
 \end{equation*}
 in which $\lambda,\mu\geq 0$ are the so called Lam\'{e} coefficients of the material. This leads to $G^{a,b}$ metric for which $a^2 = 2 \mu + \lambda$ and $b^2=\mu/2$ i.e.
 \begin{equation*}
  \frac{b}{a} = \frac{1}{2} \sqrt{\frac{2\mu}{2\mu + \lambda}} \leq \frac{1}{2}.
 \end{equation*}
It is thus interesting to note that this stronger isotropy assumption on the elastic domain imposes the constraint that $b \leq a/2$, the limiting case $b=a/2$ corresponding precisely to the square root velocity (SRV) metric of \cite{srivastava2010shape}. 
\end{remark}

\begin{remark}
\label{rem:higher_dimensions}
The previous derivations can be extended to curves in higher dimensions relatively easily. In the case of a 3D curve for example, one can introduce a tubular neighborhood of small radius $\delta$ around that curve that is again deformed by a vector field uniform in the transverse direction. Then, assuming a material with transversely isotropic elastic properties, one can show that, as $\delta$ goes to $0$, the resulting elastic energy is again given by \eqref{eq:our_metric_ab}.   
\end{remark}

\section{Proof of Theorem \ref{genSRVT} and Corollary~\ref{thm:completion}}
To derive the formula for geodesic distance \eqref{eq:dist}, we will first prove a lemma on distances in the finite-dimensional space $(\R^d\setminus \{0\},g^\lambda)$.   
 Clearly $\mathbb R^d \setminus \{0\}$ is incomplete with respect to $g^{\lambda}$ for any $\lambda>0$. We can complete it as a metric space by reinserting the origin. Let $\R^d_\lambda$ denote the metric space that is the completion with respect to the $g^{\lambda}$ metric. Note that, as a point set, $\R^d_\lambda$ is just $\R^d$, but $\R^d_\lambda$ is not a Riemannian manifold when $\lambda \neq 1$,
as the Riemannian metric $g^\lambda$ cannot be smoothly extended to the origin in this case.  We obtain the following explicit formula for the distance $d_\lambda$ on this completion:

\begin{lemma}\label{geod_dist_glambda} For 
$q_1,q_2\in\R^d_\lambda$. 
we let 
\begin{equation}
d_\lambda(q_1,q_2)=\sqrt{|q_1|^2+|q_2|^2-2|q_1||q_2|\cos(\lambda\theta)},
\end{equation}
where
\begin{equation}
\theta=\begin{cases}
\operatorname{min}\left(\cos^{-1}(q_1\cdot q_2/|q_1||q_2|),\tfrac{\pi}{\lambda}\right) &\text{ if $q_1$ and $q_2$ are both non-zero}\\
 \tfrac{\pi}{\lambda} &\text{ if $q_1$ or $q_2$ is zero}
\end{cases}
\end{equation}
Then $(\R^d_\lambda,d_\lambda)$ is the metric completion of the Riemannian manifold $(\mathbb R^d \setminus \{0\},g^{\lambda})$. 
\end{lemma}

\begin{proof} First we consider the case $d=2$.  Without loss of generality, assume that $q_1=(h,0)$ and $q_2=(k\cos\theta,k\sin\theta)$, where $h,k>0$ and $0\leq\theta\leq\pi$. (This can easily be arranged, since reflections in the $x$-axis and rotations about the origin are isometries of $\R^2_\lambda$.) Define the sector $V_\theta\subset\R^n_\lambda$ by $V_\theta=\{(r\cos\alpha,r\sin\alpha):0\leq\alpha\leq\theta\mbox{ and } r\geq0\}$. Define a function $F:V_\theta\to\R^2$ by 
\[
F(r\cos\alpha,r\sin\alpha)=(r\cos\lambda\alpha,r\sin\lambda\alpha).
\]
We now consider the case $\lambda\theta\leq\pi$, and we show that $F$ is an isometry from $V_{\theta}$ with metric given by the completion of the Riemannian manifold $(\mathring{V_{\theta}},g^{\lambda})$ into $\mathbb{R}^2$ with the Euclidian metric. In the following computations, we fix a basepoint $q=(r\cos{\alpha},r\sin{\alpha})$ in the interior $\mathring{V_{\theta}}$, and $v,w\in T_q\mathring{V_{\theta}}$. We can easily see that $F\circ O_{-\alpha}=O_{-\lambda\alpha}\circ F$ where $O_{\phi}:\R^2 \to \R^2$ is rotation by angle $\phi$. Therefore, as rotations are isometries for the Euclidian metric $g^1$,
we have
\begin{align*}
    g^1_{F(q)}(d_qF(v),d_qF(w)) &= g^1_{O_{-\lambda\alpha}\circ F(q)}(d_{F(q)}O_{-\lambda\alpha}\circ d_qF(v),d_{F(q)}O_{-\lambda\alpha}\circ d_qF(w)) \\
    &= g^1_{F \circ O_{-\alpha}(q)}(d_{O_{-\alpha}(q)} F \circ d_q O_{-\alpha}(v),d_{O_{-\alpha}(q)} F \circ d_q O_{-\alpha}(w)) \\
    &= g^1_{F(r,0)}(d_{(r,0)}F(\tilde{v}),d_{(r,0)}F(\tilde{w}))
\end{align*}
with $\tilde{v}:= d_q O_{-\alpha}(v) = O_{-\alpha}(v)$ and $\tilde{w}:= d_q O_{-\alpha}(w) = O_{-\alpha}(w)$. We can easily compute that $d_{(r,0)} F(\tilde{v}) = (\tilde{v}^{\top}, \lambda \tilde{v}^{\bot})$, and putting this together with the calculation above yields

\begin{align*}
    g^1_{F(q)}(d_qF(v),d_qF(w)) &= g^1_{F(r,0)}\left((\tilde{v}^{\top},\lambda\tilde{v}^{\bot}),(\tilde{w}^{\top},\lambda\tilde{w}^{\bot})\right) \\
    &=\tilde{v}^{\top}\tilde{w}^{\top} + \lambda^2\tilde{v}^{\bot}\tilde{w}^{\bot} \\
    &= g^{\lambda}_{(r,0)}(\tilde{v},\tilde{w})
\end{align*}
As $g^{\lambda}$ is invariant under rotations, we  have 
\[
g^{\lambda}_{(r,0)}(\tilde{v},\tilde{w})=g^\lambda_{O_\alpha (r,0)}(O_\alpha(v),O_\alpha(w)) = g^{\lambda}_{q}(v,w)
\]
and it follows that $F$ is a Riemannian isometric embedding from $(\mathring{V_{\theta}},g^{\lambda})$ into $(\mathbb{R}^2,g^1)$. It extends to a metric isometric embedding from $V_{\theta}$ into $\mathbb{R}^2$ by completion.

 $F$ is also injective and $F(V_\theta)$ is a convex subset of $\R^2$, so the geodesic (i.e. straight line) in $F(V_\theta)$ joining $F(q_1)$ to $F(q_2)$ remains in $F(V_\theta)$. Thus if we apply $F^{-1}$ to this straight line, we obtain a geodesic in $\R^2_\lambda$ joining $u$ to $v$. Since $F$ is an isometry, the length of the geodesic in $\R^2_\lambda$ is the same as the length of the straight line, and is given by the desired formula, as a simple application of the law of cosines in $\R^2$ shows.

In case that $\lambda\theta\geq\pi$,  let $\lambda' = \frac{\lambda \theta}{\pi}$. We have $\lambda' \geq 1$.
We consider once again $F_{\frac{\pi}{\theta}}: V_{\theta}\rightarrow V_{\pi}$:
$$
F_{\frac{\pi}{\theta}} (r \cos{\alpha}, r \sin{\alpha}) = (r \cos{\frac{\pi}{\theta}\alpha}, r \sin{\frac{\pi}{\theta}\alpha})
$$
Then $F_{\frac{\pi}{\theta}}$ is an isometry from $(V_{\theta}, g^{\lambda})$ to $(V_{\pi}, g^{\frac{\lambda \theta}{\pi}})$, and therefore :
$$
\mbox{dist}_{g^{\lambda}}\left( \left(h,0\right), \left(r \cos{\theta}, r \sin{\theta}\right) \right) = \mbox{dist}_{g^{\lambda'}}\left( (h,0), \left(-r , 0\right) \right)
$$

Because $\lambda' \geq 1$, for all $\omega \in \mathbb{R}^2$, $g^{\lambda'}(\omega, \omega) \geq |\omega'|_{\mathrm{euc}}^2$, and therefore $\mbox{dist}_{g^{\lambda'}} \geq \mbox{dist}_{\mathrm{euc}}$. But taking the straight line between $(h,0)$ and $(-r,0)$, we then have $\mbox{dist}_{g^{\lambda'}}\left( (h,0), \left(-r , 0\right) \right) \leq h+r$, and therefore $\mbox{dist}_{g^{\lambda'}}\left( (h,0), \left(-r , 0\right) \right) = h+r$. This yields again the desired formula. 

Having proved the theorem for $d=2$ it follows for general $d$, since any pair of elements $u,v$ is contained in a totally geodesic copy of $\R^2_\lambda\subset \R^d_\lambda$.
\end{proof}

\begin{proof}[Proof of Theorem \ref{genSRVT} and Corollary~\ref{thm:completion}.]
The statement that $R$ is a diffeomorphism is clear from the definition of the involved spaces; see also~\cite{bruveris2016optimal,lahiri2015precise}. It remains to show that $R$ is a Riemannian isometry. To this end, we calculate the derivative of $R$ at $c \in \mathcal I_0([0,1],\mathbb R^d)$ in the direction $h \in T_c I_0([0,1],\mathbb R^d)$ as
\[
d_cR(h) = \frac{1}{\sqrt{|c'|}}\left(h' - \frac{1}{2}\left(h' \cdot \frac{c'}{|c'|} \right) \frac{c'}{|c'|} \right).
\]
The component of $d_cR(h)$ tangential to $R(c)$ is 
\begin{align*}
\big(d_cR(h)\big)^\top &= \left(d_cR(h) \cdot \frac{c'}{|c'|}\right) \frac{c'}{|c'|} = \frac{1}{2\sqrt{|c'|}}  \left(h' \cdot \frac{c'}{|c'|} \right) \frac{c'}{|c'|} \\
&= \frac{\sqrt{|c'|}}{2} \left(\frac{1}{|c'|} h' \cdot \frac{c'}{|c'|} \right) \frac{c'}{|c'|} = \frac{\sqrt{|c'|}}{2} D_s h^\top.
\end{align*}
Similarly, the orthogonal component is given by
$
\big(d_cR(h) \big)^\bot = \sqrt{|c'|} D_s h^\bot.
$
Therefore, for $h,k \in T_c I_0([0,1],\mathbb R^d)$, we have
\begin{align*}
    4b^2 G_{R(c)}^{L^2_{\lambda}}(d_cR(h),d_cR(k)) &= 4b^2 \int_0^1 \lambda^2 |c'| \big(D_s h^\bot \cdot D_s k^\bot\big) + \frac{|c'|}{4} \big(D_s h^\top \cdot D_s k^\top\big) \; du \\
    &=  \int_0^1 a^2 \big(D_s h^\bot \cdot D_s k^\bot\big) + b^2 \big(D_s h^\top \cdot D_s k^\top\big) \; ds =  G^{a,b}_c(h,k),
\end{align*}
where, in the last line, we use that $\lambda = \frac{a}{2b}$ and $ds = |c'(u)| du$. The proves that $R$ is an isometry.

It remains to derive the geodesic distance formula. To do so, we recall a general fact about geodesics in path spaces. Let $(M,g)$ be a (finite-dimensional) Riemannian manifold and consider the space 
\[
\mathcal M :=C^\infty([0,1],M).
\]
By~\cite{bruveris20182}, a path in ${\mathcal M}$ given by $t \mapsto c_t$ is a length minimizing geodesic with respect to the $L^2$-Riemannian metric (defined by
\[
G^{L^2}_c(h,k) = \int_0^1 g_{c(u)}(h(u),k(u)) du
\]
for $h,k \in T_c \mathcal{M}$ parameterized smooth vector fields along $c$) if and only if for (almost all) fixed $u_0$, the curve given by $t \mapsto c_t(u_0)$ is a length-minimizing geodesic in $M$. Consequently, the geodesic distance between $c_0,c_1 \in \mathcal M$ is given by
\begin{equation}
\sqrt{\int_0^1 \operatorname{dist}^M(c_0(u),c_1(u))^2 du }, 
\end{equation}
where $\operatorname{dist}^M$ denotes geodesic distance in the finite-dimensional manifold $M$. 

We can now apply the above result with $(M,g) = (\mathbb R^d\setminus\{0\},g^\lambda)$. Let $c_0,c_1\in \mathcal I_0([0,1],\mathbb R^d)$ such that 
the geodesic in $(\mathbb R^d,g^\lambda)$ between $R(c_0)(u)$ and $R(c_1)(u)$ does not pass through the origin for any $u\in[0,1]$. Then the formula for the geodesic distance follows directly by the formula from Lemma \ref{geod_dist_glambda} and the above considerations -- note that in this case the minimum in the definition of $\theta$ is always given by the arccos term. This proves the local formula for the geodesic distance. To obtain the global formula one needs to smoothly
perturb any path that passes through the origin in such a way that the perturbed
path avoids the origin. It is easy to see that this is possible if $d\geq 3$, which shows the global formula for the geodesic distance. For $d=2$ and $\lambda=1$ a counterexample, i.e., two curves where the minimizing path can not be perturbed to avoid the origin,  has been constructed in \cite{bruveris2016optimal}. A similar argument works for general $\lambda$ and thus the formula for the geodesic distance is only valid locally in this case.
\end{proof}
To prove the statements on the metric completion we will first study these completions in the space of $R$-transforms, i.e., on $C^\infty(I, \mathbb R^d\setminus\{0\}$): 
\begin{lemma}\label{lem:complL2}
For $q_1,q_2\in L^2([0,1],\mathbb R^d)$ we let 
\begin{equation}
d_{G^{L^2_{\lambda}}}(q_1,q_2)=\sqrt{\int_0^1 |q_1(u)|^2+|q_2(u)|^2-2|q_1(u)||q_2(u)|\cos(\lambda\theta(u)) du},
\end{equation}
where
\begin{equation}
\theta(u)=\begin{cases}
\operatorname{min}\left(\cos^{-1}(q_1(u)\cdot q_2(u)/|q_1(u)||q_2(u)|),\tfrac{\pi}{\lambda}\right) &\text{ if $q_1(u), q_2(u) \neq 0$}\\
 \tfrac{\pi}{\lambda} &\text{ otherwise.}
\end{cases}
\end{equation}
We have the following two statements regarding the completion of the space of smooth functions: The space $\left(L^{2}([0,1],\mathbb R^d_{\lambda}),d_{G^{L^2_{\lambda}}}\right)$ is the metric completion
of the geodesic completion $C^{\infty}([0,1],\mathbb R^d_\lambda)$. 
If $d\geq 3$, then  $\left(L^{2}([0,1],\mathbb R^d_{\lambda}),d_{G^{L^2_{\lambda}}}\right)$ is also the metric completion
of $\left(C^{\infty}([0,1],\mathbb R^d \setminus \{0\}),G^{L^2_{\lambda}}\right)$. 
\end{lemma}
\begin{proof}
This follows directly from the definition of the geodesic distance on 
$C^{\infty}([0,1],\mathbb R^d \setminus \{0\})$, the proof of Theorem \ref{genSRVT}, and Lemma \ref{geod_dist_glambda}.
\end{proof}
Now corollary~\ref{thm:completion} follows from the results above and the formula for $R$.

\section{Proof of Theorem \ref{thm:existence}.}\label{appendix:existenceopen}
The main ingredient for the existence proof is the following result by Trouve and Younes, concerning the existence of minimizers for a wide class of optimization problems:
\begin{theo}[Theorem 3.1 and Prop. 5.1 in \cite{article}]\label{thm:Younes_Trouve} Let $f:[0,1]\times[0,1]\rightarrow \mathbb{R}_{\geq 0}$ be a bounded measurable function that satisfies the following condition
\begin{enumerate}
    \item[(H1)] There exists a finite family of closed segments $([a_j,b_j])_{j\in J}$ such that each of them is horizontal or vertical and $f$ is continuous on $[0,1]^2\backslash \bigcup_{j \in J}[a_j,b_j]$
\end{enumerate}
Then there exists an non-decreasing BV-function $\phi\in \mathcal D^*$ that maximizes the functional  $\phi \mapsto \int_D \sqrt{\dot{\phi}}(x) f(x,\phi(x)) d x$. 
Let $f_s$ be defined by
  \begin{equation*}
    f_s(x_0,y_0) = \lim_{\delta\rightarrow 0}\left(\inf\left\{f\left(x,y\right) | \left(x,y\right)\in [0,1]^2\backslash \bigcup_{j \in J}[a_j,b_j], |\left(x,y\right)-\left(x_0,y_0\right)|<\delta \right\} \right)
  \end{equation*}
Assume that $f_s$ satisfies in addition the condition:
\begin{enumerate}
    \item[(H2)] 
There does not exist any nonempty open vertical or horizontal segment $]a,b[$ such that $f_s$ vanishes on $]a,b[$.
\end{enumerate}
Then the optimizer $\phi^* \in \mathcal{D}^*$ is a strictly increasing homeomorphism.
\end{theo}
For the statements regarding the higher regularity case, we will in addition need the following technical result:
\begin{theo}[Theorem 3.3 in \cite{article}]\label{thm:Younes_Trouve2} Let $f$ be a nonnegative measurable function on $[0,1]^2$, and assume that $U_f : \mathcal{D}^*\rightarrow \mathbb{R}, \phi \mapsto \int_D \sqrt{\dot{\phi}(x)}f(x,\phi(x))d x$ reaches its maximal value at a strictly increasing continuous function $\phi^* \in \mathcal{D}^*$. Then for any $x_0 \in [0,1]$, if $f(x_0,\phi(x_0))>0$ and if $f$ is locally Hölder continuous, then $\phi^*$ is differentiable at $x_0$, with strictly positive derivative, and $\dot{\phi}^*$ is continuous in a neighborhood of $x_0$.
\end{theo}

\begin{proof}[Proof of Theorem~\ref{thm:existence}]
We start by proving the formula for the geodesic distance. Using the explicit formula for parametrized curves that was obtained in Theorem~\ref{thm:completion}, we can write the geodesic distance as 
\begin{align}\label{quotien_distance_BV2}
\operatorname{dist}^{\mathcal S}_{a,b}([c_1],[c_2])&=
2b\sqrt{\ell_{c_1}+\ell_{c_2}-2\sup_{\gamma_1,\gamma_2\in \bar \Gamma}\int_I\sqrt{\dot\gamma_1(u)\dot\gamma_2(u)} \tilde f_{a,b}\left(\gamma_1(u),\gamma_2(u))\right) d u},
\end{align}
where $\tilde f_{a,b} : I\times I  \rightarrow  \mathbb{R}$ is defined by
$$
\tilde f_{a,b}(x,y )=\left\{ \begin{array}{cl}
          \sqrt{|\dot{c}_1(x)| |\dot{c}_2(y)|}\cos\left(\frac{a}{2b} \cos^{-1}(\frac{\dot{c}_1(x)\cdot \dot{c}_2(y)}{|\dot{c}_1(x)| |\dot{c}_2(y)|})\right) & \mbox{ if } \frac{a}{2b} \cos^{-1}(\frac{\dot{c}_1(x)\cdot \dot{c}_2(y)}{|\dot{c}_1(x)| |\dot{c}_2(y)|}) \leq \pi   \\
          -1 & \mbox{ otherwise.}  
     \end{array} \right. 
$$
This formulation is, however, not convenient for us, as the function $\tilde f_{a,b}$ is not non-negative and thus one cannot directly apply the results of~\cite{article}. Thus we will first show that the the above optimization problem does not change when we substitute $\tilde f_{a,b}$ by the non-negative function $f_{a,b}$, as defined in~\eqref{eq:f_ab}. Since $f_{a,b} = \max{(\tilde f_{a,b},0)}$, we have
\begin{align*}
&\sup_{\gamma_1,\gamma_2\in \bar \Gamma}\int_I\sqrt{\dot\gamma_1(u)\dot\gamma_2(u)} \tilde f_{a,b}\left(\gamma_1(u),\gamma_2(u))\right) d u \\
&\qquad \qquad \leq \sup_{\gamma_1,\gamma_2\in \bar \Gamma}\int_I\sqrt{\dot\gamma_1(u)\dot\gamma_2(u)} f_{a,b}\left(\gamma_1(u),\gamma_2(u))\right) d u
\end{align*}
Let  $\gamma_1, \gamma_2 \in \bar{\Gamma}$, and let $A = \{u \in I,  \tilde f_{a,b}\left(\gamma_1(u),\gamma_2(u))\right) < 0\}$ the open set of negative parts of $\tilde f_{a,b}$. We can write $A$ as an at most countable disjoint union of open intervals $A=\bigcup_n I_n$ with $I_n = ]u_n^-,u_n^+[$. Now let us construct reparametrizations $\tilde \gamma_1, \tilde \gamma_2$ with derivatives equal to zero on $A$. We set $\tilde{\gamma}_1(u) = \gamma_1(u)$, $\tilde{\gamma}_2(u) = \gamma_2(u)$ for $u\in I\backslash A$ and :
$$
    \tilde{\gamma}_1(u) = \begin{cases}
    \gamma_1(2u-u_n^-) &\mbox{ for } u \in ]u_n^-,1/2(u_n^- +u_n^+)[ \cr
    \gamma_1(u_n^+) &\mbox{ for } u \in ]1/2(u_n^- +u_n^+), u_n^+[ 
    \end{cases}
$$
$$
    \tilde{\gamma}_2(u) = \begin{cases}
    \gamma_2(u_n^-) &\mbox{ for } u \in ]u_n^-,1/2(u_n^- +u_n^+)[ \cr
    \gamma_2(2u - u_n^+ ) &\mbox{ for } u \in ]1/2(u_n^- +u_n^+), u_n^+[ 
    \end{cases}
$$
We clearly have $\tilde{\gamma}_1,\tilde{\gamma}_2 \in \bar \Gamma$ and for $u \in I_n$, $\sqrt{\dot\gamma_1(u)\dot\gamma_2(u)} = 0$, so 
\begin{align*}
    \int_I\sqrt{\dot{\tilde{\gamma}}_1(u)\dot{\tilde{\gamma}}_2(u)} \tilde f_{a,b}\left(\tilde{\gamma}_1(u),\tilde{\gamma}_2(u))\right) d u &= \int_{I\backslash A}\sqrt{\dot\gamma_1(u)\dot\gamma_2(u)} \tilde f_{a,b}\left(\gamma_1(u),\gamma_2(u))\right) \\
    &= \int_I\sqrt{\dot\gamma_1(u)\dot\gamma_2(u)} f_{a,b}\left({\gamma}_1(u),{\gamma}_2(u))\right),
\end{align*}
which proves the equivalence of the two optimization problems.

Next we will prove the equivalent definition of the distance, where the infimum is taken over one BV function in the space $\mathcal{D}^*$. Thus we consider the two functionals
\begin{align}\label{eq:min_onefunction}
\begin{array}{ccl}
     \mathcal{D}^* &\rightarrow &\mathbb{R}  \\
     \phi & \mapsto &\int_I \sqrt{\dot{\phi}(x)} f_{a,b}(x,\phi(x)) d x
\end{array}  
\end{align}
and 
\begin{align}\label{eq:min_twofunctions}
\begin{array}{ccl}
     \bar \Gamma \times \bar \Gamma&\rightarrow &\mathbb{R}  \\
     (\gamma_1,\gamma_2) & \mapsto &\int_I \sqrt{\dot{\gamma}_1(u)\dot{\gamma}_2(u)} f_{a,b}\left(\gamma_1(u),\gamma_2(u)\right) du.
\end{array}  
\end{align}
 We will show a slightly stronger statement, namely  the following claim:\\
\emph{Claim A:
For $\phi \in \mathcal{D}^*$, there exist $\gamma_1, \gamma_2 \in \Bar{\Gamma}$, such that $$ \int_I \sqrt{\dot{\gamma}_1(u)\dot{\gamma}_2(u)} f_{a,b}\left(\gamma_1(u),\gamma_2(u)\right) d u=\int_I \sqrt{\dot{\phi}(x)} f_{a,b}(x,\phi(x)) d x .$$ 
Conversely, for $\gamma_1, \gamma_2 \in \Bar{\Gamma}$, there exists $\phi \in \mathcal{D}^*$ such that 
$$\int_I \sqrt{\dot{\phi}(x)} f_{a,b}(x,\phi(x)) d x = \int_I \sqrt{\dot{\gamma}_1(u)\dot{\gamma}_2(u)} f_{a,b}\left(\gamma_1(u),\gamma_2(u)\right) d u.$$ 
}

To show Claim A, let $\phi \in \mathcal{D}^*$ and $\mu$ the corresponding probability measure given by \ref{eq:def_mu}. By Lebesgue's decomposition theorem, we may write $\mu = \omega_\phi d x + \nu_{\phi} + \sum_{n \in \mathbb{N}}a_n\delta_{x_n}$, where $\omega_\phi dx$ is the absolutely continuous part of $\mu$ with respect to the Lebesgue measure, $\nu_\phi$ the singular continuous part, and for all $n\in \mathbb{N}$, $x_n \in I$ and $a_n\geq0$. The latter part can be seen as the (at most countable) jumps of $\phi$ located at the points $x_n$ and of amplitude $a_n$.  
We then consider the set $$C=\mbox{Graph}(\phi)\cup\bigcup_{n\in \mathbb{N}}\{x_n\}\times[\phi(x_n), \phi(x_n)+a_n]$$ \\
 Then $C$ is compact as it is clearly bounded and its closedness can be shown from the definition using the left continuity of $\phi$. Furthermore, $C$ is connected and $\mathcal{H}_1(C)<\infty$, since $\phi$ is of bounded variation and $\sum_n a_n \leq 1 < \infty$. Therefore $C$ is the image of a rectifiable curve, that can be reparametrized as an injective, Lipschitz continuous curve $\gamma$, cf. \cite[Lemmas 3.1 and 3.12]{falconer_1985}. We write
\begin{align}
    \gamma : 
\left\{\begin{array}{ccl} 
[0,1] & \rightarrow & [0,1]^2 \\
u &\mapsto & (\gamma_1(u), \gamma_2(u))
\end{array} \right. ,
\end{align}
where $\gamma_1$ and $\gamma_2$ are Lipschitz continuous, non-decreasing and differentiable almost everywhere (and we will let $\dot{\gamma}_i(u)=0$ if $\gamma_i$ is not differentiable in $u$). We then calculate:
\begin{align*}
     &\int_I \sqrt{\dot{\gamma}_1(u)\dot{\gamma}_2(u)} f_{a,b}\left(\gamma_1(u),\gamma_2(u)\right) d u \\
     &\qquad \qquad = \int_I   \sqrt{\dot{\gamma_1}(u)\dot{\gamma_2}(u)} f_{a,b}(\gamma_1(u),\gamma_2(u)) 1_{\dot{\gamma_1}(u)>0}(u)d u \\
    &\qquad \qquad = \int_I \dot{\gamma_1}(u)\sqrt{\frac{\dot{\gamma_2}(u)}{\dot{\gamma_1}(u)}} f_{a,b}(\gamma_1(u),\gamma_2(u)) 1_{\dot{\gamma_1}(u)>0}(u) d u \\
    &\qquad \qquad = \int_I \sum_{u \in \gamma_1^{-1}(x)} \sqrt{\frac{\dot{\gamma_2}(u)}{\dot{\gamma_1}(u)}} f_{a,b}(\gamma_1(u),\gamma_2(u)) 1_{\dot{\gamma_1}(u)>0}(u)d x,
\end{align*}
where $\gamma_1^{-1}(x) = \{u \in I, \gamma_1(u)=x\}$ and $1_C$ denotes the indicator function for condition $C$. The last equality follows from the area formula \cite[Theorem 3.2.3]{Federer}. Indeed, given the assumptions on $\gamma_1$, it is differentiable almost everywhere and we have by (\cite{mattila1999geometry}, p.103) that $\{\gamma_1(u), \dot{\gamma}_1(u) = 0 \}$ is of Lebesgue measure zero. Thus, for almost all $x \in [0,1]$, $\gamma_1^{-1}(x)$ is reduced to a single point. Then, by setting $x=\gamma_1(u)$, we have $(\gamma_1(u),\gamma_2(u)) = (x, \phi(x))$ with $\phi(x)=\gamma_2(\gamma_1^{-1}(x))$.

It follows that for almost all $u$ such that $\dot{\gamma_1}(u)>0$, one has:
$$\dot{\gamma}_2(u) = \frac{d}{du}(\phi\circ\gamma_1(u)) = \dot{\phi}(\gamma_1(u))\dot{\gamma_1}(u) = \dot{\phi}(x) \dot{\gamma_1}(u),$$ and thus $\dot{\phi}(x) = \frac{\dot{\gamma_2}(u)}{\dot{\gamma_1}(u)}$. 
Going back to the original equality, this leads to 
\begin{align*}
    \int_I \sqrt{\dot{\gamma}_1(u)\dot{\gamma}_2(u)} f_{a,b}\left(\gamma_1(u),\gamma_2(u)\right) d u &= \int_I \sum_{u \in \gamma_1^{-1}(x)} \sqrt{\frac{\dot{\gamma_2}(u)}{\dot{\gamma_1}(u)}} f_{a,b}(\gamma_1(u),\gamma_2(u)) 1_{\dot{\gamma_1}(u)>0}(u)d x \\
    &= \int_I  \sqrt{\dot{\phi}(x)} f_{a,b}(x,\phi(x)) d x,
\end{align*} 
which proves the first direction of Claim A. 

To prove the converse direction of Claim A, we let $\gamma_1,\gamma_2 \in \Bar{\Gamma}$, where we can choose $\gamma_1, \gamma_2$, up to a reparametrization, to be Lipschitz continuous.
We consider the generalised inverse $\gamma_1^{-} \in \mathcal{D}^*$, and let $\phi = \gamma_2\circ\gamma_1^{-}$.
By \cite[Lemma 5.8]{article} the generalized inverse is again an element of $\mathcal{D}^*$. Since $\gamma_2$ is Lipschitz continuous, and since composition with Lipschitz functions keeps  $\mathcal{D}^*$ invariant,
we have that $\phi \in \mathcal{D}^*$, see e.g. \cite[Theorem 4]{josephy}. Now one can obtain the desired equality
$$\int_I \sqrt{\dot{\phi}(x)} f_{a,b}(x,\phi(x)) d x = \int_I\sqrt{\dot{\gamma}_1(u)\dot{\gamma}_2(u)} f_{a,b}\left(\gamma_1(u),\gamma_2(u)\right) d u,$$
by a similar computation as above, which concludes the proof of Claim A.

Now the first statement  of item 1 
-- $a<b$ and $c_1,c_2\in PC^1(I,\mathbb R^d)$-- follows directly from Theorem~\ref{thm:Younes_Trouve}: since $c_1$ and $c_2$ are assumed to be piecewise $C^1$, the function $f_{a,b}$ is bounded and continuous when $\dot{c}_1$ and $\dot{c}_2$ are continuous. For $x$ (resp. $y$) a point of discontinuity of $\dot{c}_1$ (resp. $\dot{c}_2$), $f_{a,b}$ is not continous on the vertical segment $\{x\}\times [0,1]$ (resp. the horizontal segment $[0,1]\times \{y\}$. Thus $f_{a,b}$ satisfies (H1). For $a<b$ we have in addition that $f_{a,b}>c$ where $c>0$ and thus we also have $f_s>c$ does not vanish. By Theorem~\ref{thm:Younes_Trouve} this implies that the minimizer exists in $\mathcal D^*$ and is a strictly increasing homeomorphism.  

It remains to prove the statements assuming additional smoothness of the curves $c_1$ and $c_2$, namely that
$\dot c_i$ are Lipschitz continuous.  Therefore we show that in this case the function $f_{a,b}$ is also Lipschitz continuous: the application $\theta \mapsto \cos(\frac{a}{2b}\cos^{-1}(\theta))$ is differentiable on $]-1,1[$, and its derivative is bounded, therefore $\theta \mapsto \cos(\frac{a}{2b}\cos^{-1}(\theta))$ is Lipschitz continuous on $[-1,1]$. As $\dot{c}_1$ and $\dot{c}_2$ are Lipschitz continuous, the  function $x,y \mapsto \frac{\dot{c}_1(x)\cdot\dot{c}_2(y)}{|\dot{c}_1(x)| |\dot{c}_2(y)|}$ is also Lipschitz continuous. Therefore by composition, $f_{a,b}$ is Lipschitz continuous and thus also H\"older continuous since we are working on a compact domain. We have already shown that the minimizer $\phi$ is strictly increasing and continuous. As $f_{a,b}$ is strictly positive everywhere we obtain by Theorem~\ref{thm:Younes_Trouve2} that $\phi$ is of class $\mathcal{C}^1$ on all of $I$, which concludes the proof of the second statement of point 1.

For the second item, $a\geq b$, there may exist areas where $f_{a,b}=0$, which leads to optimizers that have jumps and are thus not continuous. To deal with these difficulty, we will follow the same approach as in~\cite{bruveris2016optimal} and consider a pair of generalized reparametrization functions, that might have vertical parts but no jumps. Using Claim A we can still focus on maximizing~\eqref{eq:min_onefunction} on the space of BV functions, which allows us to use again the result of Trouvé and Younes~\cite{article}. In particular by Theorem~\ref{thm:Younes_Trouve}, cf. \cite[Proposition 5.1]{article}, there exists $\phi \in \mathcal{D}^*$ that maximizes~\eqref{eq:min_onefunction}, and thus by Claim A there exist $\gamma_1,\gamma_2 \in \bar{\Gamma}$ such that $\operatorname{dist}^{\mathcal S}_{a,b}([c_1],[c_2]) = \operatorname{dist}_{a,b}(c_1\circ\gamma_1,c_2\circ \gamma_2)$, which proves the existence result in item 2.

Finally, we shall construct a counter-example when $a>b$ if the curves $c_1$ and $c_2$ are only in the space $AC(I,\mathbb{R}^d)$. To that end, we adapt the counter-example from~\cite[Section 6]{bruveris2016optimal}.
Let $0<\epsilon<\frac{1}{6}$ and define 
\[
v_1(t) = 
    \begin{pmatrix} 
      \cos{\frac{2a\pi}{b}\epsilon  t}\\ 
      \sin{\frac{2a\pi}{b}\epsilon  t}
    \end{pmatrix}, \quad v_2 = \begin{pmatrix} 
      \cos{\frac{4a\pi}{3b}}\\ 
      \sin{\frac{4a\pi}{3b}}
    \end{pmatrix}, \quad \mbox{and} \quad v_3 = \begin{pmatrix} 
      \cos{\frac{4a\pi}{3b}}\\ 
      -\sin{\frac{4a\pi}{3b}}
    \end{pmatrix}.
\]
Then we have that $ \frac{b}{2a}\cos^{-1}(\frac{v_i(t)\cdot v_j(t)}{|v_i(t)| |v_j(t)|}) \geq \frac{\pi}{2}$ for each $i\neq j$, and therefore $f_{a,b}\leq 0$ for those vectors. We define two curves $c_1, c_2 \in AC(I, \mathbb{R}^d$) such that :
    \begin{align*}
        \dot{c}_1(u) &= v_1(u) 1_A(u) + v_2 1_B(u) \\
        \dot{c}_2(u) &= v_1(u) 1_A(u) + v_3 1_B(u)
    \end{align*}
where $B \subset I$  a modified Cantor set such that $B$ is closed, nowhere dense with $\lambda(B) = \frac{1}{2}$, and $A=I\backslash B$. Following the same proof as in \cite[section 6]{bruveris2016optimal}, we have that $$\sup_{\gamma_1,\gamma_2\in \bar \Gamma}\int_If_{a,b}\left(\gamma_1(u),\gamma_2(u)\right) d t = \lambda(A)$$ and is not attained. 
\end{proof}

\section{Proof of Theorem \ref{thm:existenceclosed}}\label{appendix:closed}
\begin{proof}
For $c_1, c_2\in PC^1(S^1,\mathbb R^d)$ and $a,b > 0$ we define the functional 
\begin{align*}
    F : \left\{\begin{array}{ccl}
         S^1 & \rightarrow &\mathbb{R}  \\
         \tau & \mapsto & \inf_{\gamma_1,\gamma_2\in \bar \Gamma}\operatorname{dist}_{a,b}(c_1\circ\gamma_1,c_2\circ S_{\tau}\circ \gamma_2)
    \end{array}\right.
\end{align*}
We aim to show that $F$ is continuous, which will directly lead to the desired conclusion. Therefore let $\epsilon>0, \tau, \tau' \in S^1$. Then
\begin{align*}
   | F( &\tau) - F(\tau')|=
   \lvert\inf_{\gamma_1,\gamma_2\in \bar\Gamma}\operatorname{dist}_{a,b}(c_1\circ\gamma_1,c_2\circ S_{\tau}\circ\gamma_2) - \inf_{\gamma_1',\gamma_2'\in \bar\Gamma}\operatorname{dist}_{a,b}(c_1\circ\gamma_1',c_2\circ S_{\tau'}\circ\gamma_2')\rvert \\
    &=\lvert \operatorname{dist}^{\mathcal S}_{a,b}([c_1],[c_2\circ S_{\tau}]) - \operatorname{dist}^{\mathcal S}_{a,b}([c_1],[c_2\circ S_{\tau'}])\lvert \leq \operatorname{dist}^{\mathcal S}_{a,b}([c_2\circ S_{\tau}],[c_2\circ S_{\tau'}] \\
    &\leq \inf_{\gamma,\gamma'\in \bar\Gamma}\operatorname{dist}_{a,b}(c_2\circ S_{\tau}\circ\gamma,c_2\circ S_{\tau'}\circ\gamma') = \inf_{\gamma,\gamma'\in \Gamma}\operatorname{dist}^{L^2_{a/(2b)}}\left(Q(c_2\circ S_{\tau}\circ\gamma),Q(c_2\circ S_{\tau'}\circ\gamma')\right) \\
    &\leq \operatorname{dist}^{L^2_{a/(2b)}}\left(Q(c_2\circ S_{\tau}),Q(c_2\circ S_{\tau'})\right).
\end{align*}

By definition, we have that $G^{L^2_{a/(2b)}}\leq \max (\frac{a}{2b},1)^2 \ G^{L^2}$, therefore we can deduce that
$$\operatorname{dist}^{L^2_{a/(2b)}}\left(Q(c_2\circ S_{\tau}),Q(c_2\circ S_{\tau'})\right)\leq \max{\left(\frac{a}{2b},1\right)}||Q(c_2\circ S_{\tau})- Q(c_2\circ S_{\tau'}) \|_{L^2}.$$
Since the space $C(I, \mathbb{R}^d)$ is dense in $L^2(I, \mathbb{R}^d)$, we can choose $g\in C(I, \mathbb{R}^d)$ such that $\|Q(c_2) -g\|_{L^2} \leq \epsilon/3$. By change of variable, we also have $\|Q(c_2)\circ S_{\tau} -g\circ S_{\tau}\|_{L^2} \leq \epsilon/3$ and $\|Q(c_2)\circ S_{\tau'} -g\circ S_{\tau'}\|_{L^2} \leq \epsilon/3$. Since $g$ is continuous and $I$ is compact, $g$ is also uniformly continuous by the Heine-Borel theorem. Thus we have, for $|\tau-\tau'|$ small enough,
$$
\|g\circ S_{\tau} - g\circ S_{\tau'} \|_{L^2} \leq \epsilon/3 
$$
Finally we have:
\begin{align*}
  \left| F(\tau) - F(\tau')\right| &\leq \max{\left(\frac{a}{2b},1\right)}\|Q(c_2\circ S_{\tau})- Q(c_2\circ S_{\tau'}) \|_{L^2} \\
  &\leq \max{\left(\frac{a}{2b},1\right)}\left(\|Q(c_2\circ S_{\tau})- g\circ S_{\tau} \|_{L^2}\right. \\
  &\qquad \qquad \qquad + \left.\|g\circ S_{\tau}-g\circ S_{\tau'}\|_{L^2} + \|g\circ S_{\tau'}-Q(c_2\circ S_{\tau'} \|_{L^2} \right) \\
  &\leq \max{\left(\frac{a}{2b},1\right)}(\epsilon/3+\epsilon/3+\epsilon/3) = \frac{a}{2b} \epsilon
\end{align*}
Thus we have shown that $F$ is continuous function on the compact set $S^1$. Consequently there exists an optimal $\tau \in S^1$ such that
$F(\tau) = \inf_{S^1}F
$.
Now the remaining statement follows directly from Theorem~\ref{thm:existence}.
\end{proof}
\end{document}